\documentclass[11pt]{amsart}
\usepackage{amssymb}
\usepackage{amsfonts}
\usepackage{amscd}
\usepackage{amsmath}
\usepackage{mathrsfs}
\usepackage{curves}
\usepackage{epsfig}
\usepackage[pass]{geometry}
\usepackage{graphicx}
\usepackage{verbatim}
\usepackage{latexsym}
\usepackage{eucal}

\numberwithin{equation}{section}

\newtheorem{theorem}{Theorem}[section]
\newtheorem{lemma}[theorem]{Lemma}
\newtheorem{prop}[theorem]{Proposition}

\newtheorem{definition}[theorem]{Definition}

\def \mca {{\mathscr A}}

\def \mcd {{\mathscr D}}
\def \mce {{\mathscr E}}
\def \mcf {{\mathscr F}}
\def \mcg {{\mathscr G}}
\def \mch {{\mathscr H}}
\def \mci {{\mathscr I}}

\def \mco {{\mathscr O}}
\def \mcp {{\mathscr P}}

\def \mcs {{\mathscr S}}
\def \mct {{\mathscr T}}

\def \mcz {{\mathscr Z}}
\def \mcy {{\mathscr Y}}

\def \mbr {{\mathbb R}}

\def \id {\operatorname{Id}}

\def \tr {\operatorname{Tr}}

\def \diag{\textrm{Diag}}

\def \beqq {\begin{equation}}
\def \eeqq {\end{equation}}

\def \bpf {\begin{proof}}
\def \epf {\end{proof}}
\def \beq {\begin{equation*}}
\def \eeq {\end{equation*}}

\def \eps {\epsilon}   
\def \la {\lambda}   
\def \La {\Lambda}    
   
\def \lap {\Delta}
\def \p {\partial}
\def \ha {\frac{1}{2}}

\def \tilde {\widetilde}

\def \uone {u^{(1)}}
\def \utwo {u^{(2)}}
\def \fone  {f^{(1)}}
\def \ftwo  {f^{(2)}}

\begin{document}
 \title{Nonlinear interaction of waves in elastodynamics and an inverse problem}
\author{Maarten de Hoop}
\address{Maarten de Hoop
\newline
\indent Computational and Applied Mathematics and Earth Science, Rice University}
\email{mdehoop@rice.edu}

\author{Gunther Uhlmann} 
\address{Gunther Uhlmann
\newline
\indent Department of Mathematics, University of Washington,  
\newline
\indent and Institute for Advanced Study, the Hong Kong University of Science and Technology}
\email{gunther@math.washington.edu}

\author{Yiran Wang}
\address{Yiran Wang
\newline
\indent Department of Mathematics, University of Washington 
}
\email{wangy257@math.washington.edu}
 \begin{abstract}
We consider nonlinear elastic wave equations generalizing Gol'dberg's five constants model. We analyze the nonlinear interaction of two distorted plane waves and characterize the possible nonlinear responses. Using the boundary measurements of the nonlinear responses, we solve the inverse problem of determining elastic parameters from the displacement-to-traction map.   
\end{abstract}
\maketitle

\section{Introduction}

\subsection{The nonlinearity in elastodynamics}
We introduce the nonlinear elastic system to be studied in this work. Our model is a generalization of the five constant model widely used in the literature since the work of Gol'dberg \cite{Gol}. We shall follow the presentation in Landau-Lifschitz \cite{LL}.  The materials are classical, however we would like to review its derivation to show the sources and significance of the nonlinearity in elastodynamics. 

Consider an elastic body occupying an open bounded region $\Omega \subset \mbr^3$ with smooth connected boundary $\p \Omega$. The closure is denoted by $\overline\Omega$. We denote points in $\mbr^3$ by $x = (x_1, x_2, x_3)$. When the body is deformed, the distances between points are changed. Suppose that point $x\in \Omega$ is displaced to $x' = (x_1', x_2', x_3') \in \mbr^3$ and the displacement vector is $u = x' - x$. The length element $dl = (dx_1 + dx_2 + dx_3)^\ha$ is changed to $dl'= (dx'_1 + dx'_2 + dx'_3)^\ha$ and 
\beq
(dl')^2 = dl^2 + 2 e_{ik} dx_i dx_k,
\eeq
where $e_{ik}$ is the {\em strain tensor} defined by
\beqq\label{eqstrain}
e_{mn} =  \ha (\frac{\p u_m}{\p x_n} + \frac{\p u_n}{\p x_m} + \frac{\p u_k}{\p x_m}\frac{\p u_k}{\p x_n}).
\eeqq
Hereafter, the Einstein summation convention is used. The strain tensor describes the changes in an element of length when the body is under deformation. For small deformations, one   ignores the quadratic terms and take 
\beq
\tilde e_{mn} = \ha (\frac{\p u_m}{\p x_n} + \frac{\p u_n}{\p x_m})
\eeq
as an approximation of $e_{mn}. $ This is the strain tensor used in linearized elasticity. 

We only consider the thermostatic state of the body so that the free energy  $\mce$ of the body is a scalar function of the strain tensor only, namely $\mce = \mce(e_{ik})$. For an isotropic elastic medium, we can express $\mce$ in terms of the invariants $\tr(e), \tr(e^2), \tr(e^3)$ etc. For small deformation, one  expand $\mce$ up to quadratic terms in $\nabla u$ to get
\beq
\begin{gathered}
\mce
= \mce_0 + \ha \la(x) (\tr \tilde e)^2 + \mu(x) \tr(\tilde e^2) 
= \mce_0 + \ha \la(x)  (\tilde e_{ii})^2 + \mu(x) \tilde e_{ik}^2,
\end{gathered}
\eeq
where $\mce_0$ is a constant and $\la, \mu$ are called Lam\'e coefficients. Note that the $\tilde e_{ik}$ above are indeed $e_{ik}$ as the higher order terms are ignored. The {\em stress tensor} is given by
\beqq\label{eqlinstress}
\tilde S_{mn} = \dfrac{\p \mce}{\p \tilde e_{mn}} = \la(x) \tilde e_{ii} \delta_{mn} + 2\mu(x) \tilde e_{mn}.
\eeqq
To show the dependence of $\tilde S$ on $x\in \mbr^3$ and $u$, we also use the notation $\tilde S(x, u)$. The stress tensor is related to the internal force $T$ of the body under deformation via  $T = \nabla\cdot \tilde S.$ 
Now using Newton's second law, we obtain the differential equation describing the deformation of the body
\beqq\label{eqlin0}
\rho \frac{\p^2 u}{\p t^2} = \nabla\cdot \tilde S(x, u) + F,
\eeqq
where $F = (F_1, F_2, F_3)\in \mbr^3$ is an (external) force on the body (e.g.\ the gravity) and $\rho$ is the density of the elastic medium.  Actually, we just derived the linearized elastic wave equation.

Now we take into account the nonlinear effects. We expand the energy density $\mce$ to cubic terms
\beq
\begin{gathered}
\mce = \mce_0 + \ha \la(x) (\tr e)^2 + \mu(x) \tr (e^2) + \frac 13 A(x)\tr(e^3) + B(x)\tr(e^2)\tr(e) + \frac 13 C(x) (\tr e)^3\\
= \mce_0 + \ha \la(x) (e_{ii})^2 + \mu(x) e_{ik}^2 + \frac 13 A(x) e_{ik}e_{il}e_{kl} + B(x) e_{ik}^2 e_{ll} + \frac 13 C(x) (e_{ll})^3,
\end{gathered}
\eeq
 see Landau-Lifschitz \cite[Section 26]{LL}. In the reference, $\la, \mu, A, B, C$ are all constants so the model is called the five constant model. Other equivalent forms in the literature and their relations can be found in Norris \cite{No}. Here, we consider a more general model in which all the parameters are smooth functions on $\overline\Omega$. In the expression of $\mce$, we should use the strain tensor in \eqref{eqstrain} and keep the nonlinear terms. We consider the tensor defined as
\beqq\label{eqsig}
\begin{gathered}
S_{mn} = \dfrac{\p \mce}{\p  (\p u_m/ \p x_n)} =  \la(x)  e_{jj}  (\delta_{mn} + \frac{\p u_m}{\p x_n}) + 2\mu(x) (e_{nm} + e_{nj} \frac{\p u_m}{\p x_j}) \\
+ A(x)e_{mj}e_{nj} + B(x)(2e_{jj}e_{mn} + e_{ij}e_{ij}\delta_{mn}) + C(x)e_{ii}e_{jj} \delta_{mn}, \ \ m, n = 1, 2, 3. 
\end{gathered}
\eeqq
This tensor is no longer the stress tensor and it is not symmetric. However, the quantity $\nabla\cdot S$ still gives the internal force, hence we again get the dynamical equation of the same form   
 \beqq\label{eqnon0}
 \rho \frac{\p^2 u}{\p t^2}  =  \nabla\cdot  S(x, u) +  F.
\eeqq
This is the nonlinear elastic equation we study in this work. We point out that the nonlinearity of the system comes from two  sources: the higher order expansion of the free energy $\mce$ and the nonlinear term in the strain tensor.

\subsection{The interaction of two waves}
We consider the initial boundary value problem for \eqref{eqnon0}:
\beqq\label{eqnon}
\begin{gathered}
 \frac{\p^2 u(t, x)}{\p t^2} -  \nabla\cdot  S(x, u(t, x)) = 0, \ \ (t, x)\in \mbr\times \Omega,\\
 u(t, x) = f(t, x), \ \  (t, x)\in \mbr_+\times \p \Omega,\\
 u(t, x)  =  0, \ \ (t, x)\in \mbr_-\times \Omega,
\end{gathered}
\eeqq
where $S(x, u(t, x))$ is given by \eqref{eqsig}. Throughout this work, we assume that $\la, \mu, A, B, C$ are smooth functions on $\mbr\times \overline\Omega$. Here, for simplicity, we took $\rho = 1$. We know (see e.g. \cite{SD}) that upon changing variables and introducing lower order terms, the system \eqref{eqnon0} can always be reduced to $\rho=1$. Also, we took $F = 0$ in \eqref{eqnon0}. It is easy to see that $u = 0$ is a trivial solution to the problem if $f = 0$. Later, we also use $Z = \mbr\times \overline\Omega$ and $Y = \mbr \times \p\Omega$. 

The equation in \eqref{eqnon} is a second order quasilinear system. In general, the solution may develop shocks and we do not expect long time existence result. We establish the well-posedness for small boundary data in Section \ref{sec-well}. The novelty of this work is that we analyze the nonlinear interactions of two (distorted) plane waves and show that certain nonlinear responses are generated and they carry the information of the nonlinear parameters. More precisely, let the boundary sources $f$ be
\beq
f = \eps_1 \fone + \eps_2 \ftwo 
\eeq
depending on two small parameters $\eps_1, \eps_2$. The solution $u$ of \eqref{eqnon} with boundary source $f$ has an asymptotic expansion
\beq
u = \eps_1 \uone + \eps_2 \utwo + \eps_1^2 u^{(11)} + \eps_2^2 u^{(22)} + \eps_1\eps_2 u^{(12)} + \text{higher order terms in $\eps_1, \eps_2$}.
\eeq
Here, $\uone, \utwo$ are linear responses satisfying the linearized equations 
 \beqq\label{eqlin1}
 \begin{split}
 P u^{(\bullet)}(t, x)
 &= 0, \ \ (t, x)\in \mbr\times  \Omega,\\
 u^{(\bullet)}(t, x) &= f^{(\bullet)}(t, x), \ \ (t, x) \in \mbr_+\times \p \Omega, \\ 
u^{(\bullet)} (t, x) &= 0,   \ \ (t, x) \in \mbr_-\times \Omega,
\end{split}
 \eeqq
where $\bullet  = 1, 2$ and $u^{(11)}, u^{(12)}, u^{(22)}$ are nonlinear responses satisfying 
\beqq\label{eqlin2}
\begin{split}
 Pu^{(ij)}(t, x) 
 &= \nabla \cdot \mcg(u^{(i)}, u^{(j)}), \ \ (t, x) \in \mbr\times  \Omega,\\
 u^{(ij)}(t, x) &= 0, \ \ (t, x) \in \mbr_+\times \p \Omega, \\ 
u^{(ij)} (t, x) &= 0, \ \ (t, x)\in \mbr_-\times  \Omega,
\end{split}
 \eeqq
where $i, j \in \{1, 2\}$ and the term $\mcg$ is quadratic in $u^{(i)}, u^{(j)}$ and comes from the nonlinear terms of \eqref{eqnon}, see \eqref{eqlin3} for its exact form.  

The nonlinear interactions of elastic waves are of great interest in seismology, rock sciences etc. In the literatures e.g.\ \cite{Gol, JK, KSC} among many others, they have been mostly analyzed by taking $\uone, \utwo$ as (smooth) plane waves of the form
\beq
 e^{\imath(- tw + \vec k \cdot x)} \vec a, 
\eeq
 where $\imath^2 = -1$ and $\vec a, \vec k \in \mbr^3$ are the polarization vector and wave vector respectively. The nonlinear responses are recognized as sum or difference harmonics. One disadvantage of the plane wave approach is that the plane waves extend to the whole space hence it becomes difficult to localize the nonlinear interactions.  We shall use distorted plane waves propagating near fixed directions. Locally, they can be expressed as oscillatory integrals of the form
\beq
\int e^{\imath (t, x)\cdot \xi} \vec a(t, x; \xi) d\xi,
\eeq
where the amplitude $\vec a(t, x; \xi)$ belongs to some symbol spaces. The waves and nonlinear responses are characterized using their wave front sets.  We construct proper sources $f^{(\bullet)}$ so that $u^{(\bullet)}$ are conormal distributions. This is done in Section \ref{sec-lin} using microlocal constructions for the initial boundary value problem. The conormal distributions appear frequently in applications,  such as Heaviside functions and impulse functions, see \cite{Ho3} for more examples. Next, we show in Theorem \ref{thmresp} that the nonlinear interactions of $\uone, \utwo$ generates new singularities in $u^{(12)}$. Because of the P--S wave decomposition, there are many cases of the interaction. We are able to determine all the possible responses and find conditions when the responses are non-trivial. The results are summarized in Table \ref{tabinter} in Section \ref{sec-non}.

\subsection{The inverse problem}
 Our next goal  is to determine the elastic parameters from the boundary measurements of the nonlinear responses. We introduce notions to state the result. 
For the linearized equations 
\beq
Pu = \frac{\p^2 u}{\p t^2} - \nabla\cdot \tilde S(x, u) = 0,
\eeq
where $\tilde S(x, u)$ is defined in \eqref{eqlinstress}, the characteristic variety of $P$ is the union of sub-varieties 
 \beqq\label{eqchar}
 \begin{gathered}
 \Sigma_P = \{(\tau, \xi)\in T^*(\mbr\times \overline\Omega):  \tau^2 - \langle \xi, \xi\rangle_P=0\}, \ \  \langle \xi, \xi\rangle_P = (\la(x) + 2\mu(x))|\xi|^2, \\
  \Sigma_S = \{(\tau, \xi)\in T^*(\mbr\times \overline \Omega):  \tau^2 - \langle \xi, \xi\rangle_S =0\},\ \  \langle \xi, \xi\rangle_S = \mu(x) |\xi|^2,
 \end{gathered}
 \eeqq
 which corresponds to shear and compressional waves. We assume that 
 \beqq\label{eqassum}
 \la+ \mu > 0, \quad \mu > 0 \text{ on $\overline\Omega.$}
 \eeqq
 Then the operator $P$ is a system of real principal type (in the sense of Denker \cite{Den}), see \cite[Prop. 4.1]{HU}. We let $g_{P/S}$ be the Riemannian metric on $\overline{\Omega}$ corresponding to $\langle\cdot\rangle_{P/S}$ and let $\text{diam}_{P/S}(\Omega)$ be the diameter of $\Omega$ with respect to $g_{P/S}$. We notice that $\text{diam}_{S}(\Omega) > \text{diam}_P(\Omega)$ in view of \eqref{eqchar} and \eqref{eqassum}.

Using the well-posedness result established in Section \ref{sec-well}, we define the displacement-to-traction map as follows. For any fixed $T_0>0$, we show in Theorem \ref{thmwell} that there exits $\eps_0> 0$ so that for any $f\in C^m([0, T_0]\times \p\Omega)$ supported away from $t=0$ and $f$ sufficiently close to the zero function, there exists a unique solution $u(t, x)$ of \eqref{eqnon}. Then we define the displacement-to-traction map as
\beq
\La_{T_0}: f = u|_{[0, T_0] \times \p \Omega} \rightarrow \nu \cdot S(x, u) |_{[0, T_0] \times \p \Omega}, 
\eeq
where $\nu = \nu(x)$ is the exterior normal to $\p \Omega$. We also use $\La$ for $\La_{T_0}$ when $T_0$ is clear from the context. 

\begin{theorem}\label{main}
Assume that $\p \Omega$ is strictly convex with respect to $g_{P/S}$ and there is no conjugate point for $g_{P/S}$ in $\overline\Omega$. For $T_0> 2\text{diam}_{S}(\Omega)$, the parameters $\la, \mu, A, B$ are uniquely determined in $\overline \Omega$ by $\La_{T_0}$.
\end{theorem}

It is worth mentioning that the linear version of Theorem \ref{main} has been extensively studied in the literature. In particular, for the isotropic elastic equations, it is proved in \cite{Ra} and \cite{HU} that the P/S wave speeds (hence the Lam\'e parameters) are uniquely determined by the displacement-to-traction map. Because the linearized problem in our model is isotropic, the main interest here is to determine the nonlinear parameters. We also remark that our proof leads to an explicit way to reconstruct the nonlinear parameters from the measurement with properly chosen boundary sources. Also, we prove in Prop.\ \ref{lmC} that the parameter $C$ cannot be determined at least from the leading term of the generated nonlinear responses. However, it is likely in view of the work \cite{LUW1} that $C$ can be determined from the interaction of three or more waves.  This is not pursued in this work.

\section{The well-posedness for small boundary data}\label{sec-well}
We establish the well-posedness of the initial boundary value problem \eqref{eqnon} which we recall below
\beq 
\begin{gathered}
\frac{\p^2 u(t, x)}{\p t^2}  -  \nabla\cdot  S(x, u(t, x)) = 0, \ \ (t, x) \in \mbr\times \Omega,\\
 u(t, x) = f(t, x), \ \  (t, x) \in \mbr_+ \times \p \Omega,\\
 u(t, x)  =  0, \ \ (t, x) \in \mbr_-\times \Omega,
\end{gathered}
\eeq 
where 
\beq
\begin{gathered}
S_{mn}(x, u) =  \la(x)  e_{jj}  (\delta_{mn} + \frac{\p u_m}{\p x_n}) + 2\mu(x) (e_{nm} + e_{nj} \frac{\p u_m}{\p x_j}) \\
+ A(x)e_{mj}e_{nj} + B(x)(2e_{jj}e_{mn} + e_{ij}e_{ij}\delta_{mn}) + C(x)e_{ii}e_{jj} \delta_{mn}. 
\end{gathered}
\eeq
In the literature, the well-posedness of quasilinear hyperbolic systems are studied for the initial value problem ($\Omega = \mbr^3$) in \cite{HKM} with applications to nonlinear elastodynamics and general relativity. Some variants of the results are obtained by Kato for scalar equations or other initial-boundary conditions. Dafermos and Hrusa studied the initial-boundary value problem for nonlinear elastic equations in \cite{DH} which applies to our model. However, only the short time existence result was established for the Dirichlet boundary conditions. Their result is close to what we need. We shall modify their proof to obtain our result. We refer to \cite{NW} for similar treatments for one dimensional scalar wave equations. 

We denote the $L^p$ based Sobolev space on $\Omega$ of order $m$ by $W^{m, p}(\Omega; \mbr)$.  The compactly supported Sobolev functions are denoted by $W_0^{m,p}(\Omega, \mbr)$. When $p=2$, we also use $H^m(\Omega) = W^{m, 2}(\Omega; \mbr), H^m_0(\Omega) = W^{m, 2}_0(\Omega; \mbr)$. For $f\in C^m(M), M\subset \mbr^4$, we denote the semi-norm by
\beq
\|f\|_{C^m(M)} \doteq \sup_{x\in M} \sum_{|\alpha|\leq m} |\p^\alpha_x f(x)|. 
\eeq

The main result of this section is 
\begin{theorem}\label{thmwell}
Let $T_0 > 0$ be fixed. Assume that $f\in C^m([0, T_0]\times \p \Omega), m\geq3$ is supported away from $t=0$. Then there exists $\eps_0>0$ such that for $\|f\|_{C^m}< \eps_0$, there exists a unique solution 
\beq
u \in \bigcap_{k = 0}^m C^k([0, T_0]; W^{m-k, 2}(\Omega, \mbr))
\eeq
to \eqref{eqnon} and we have the estimates 
\beq
\max_{t\in [0, T_0]}\|\p^{m-k}_t u(t)\|_{W^{m-k, 2}(\Omega)} \leq C_0 \|f\|_{C^m(\mbr\times \p \Omega)}, 
\eeq
where $C_0>0$ does not depend on $f$. 
\end{theorem}

We make several remarks. We formulate and prove the result specifically suited to our need. The assumption that $f$ is supported away from $t=0$ is for simplicity.  In general, the theorem should work if $f$ satisfies certain compatibility conditions at $\{0\}\times \p \Omega$ with the initial conditions.  The proof of the theorem is based on some modifications of \cite[Theorem 5.2]{DH}. Indeed, the proof in \cite{DH} is quite involved and was build upon an abstract framework.  To minimize the amount of additional work, we will follow \cite{DH} very closely, even their notations. We remark that we have not tried to get sharp results which are not necessary for the inverse problem.

 \bpf[Proof of Theorem \ref{thmwell}]
The first step is to convert the problem to a Dirichlet problem. Suppose that $f\in C^m(\mbr\times \p \Omega)$ with $m\geq 3$ and $f$ is compactly supported in $t> 0.$ We use the Seeley extension, see \cite[Section 1.4]{Me}. Following the arguments there, we can find a function $\widetilde f\in C^m(\mbr\times \Omega)$ such that $\widetilde f|_{\mbr\times \p \Omega} = f$ and $\widetilde f$ is supported in $t> 0$. Moreover, the extension is continuous namely,
\beq
\|\widetilde f\|_{C^m(\mbr\times \Omega)} \leq C \|f\|_{C^m(\mbr\times \p \Omega)}.
\eeq
Hereafter, $C$ denotes a generic constant. Let $u = \widetilde u + \widetilde f$. We have 
\beq
S(u) =   \mca(\widetilde u, \widetilde f) + S(\widetilde f),
\eeq
where 
\beq
 \mca_{mn}(\widetilde u, \widetilde f) = \dfrac{\p \mce(\widetilde u+ \widetilde f)}{\p (\p \widetilde u_m/\p x_n)} = S(\widetilde u) +  \mci(t, x, \widetilde u, \nabla \widetilde u, \widetilde f, \nabla \widetilde f). 
\eeq
Here, $\mci$ is a smooth function of its arguments and we recall that  $\mce$ is the scalar energy function. We can further write 
\beq
\begin{split}
\big(\nabla\cdot \mca(\widetilde u, \widetilde f) \big)_i &= \sum_{j, \alpha, \beta = 1}^3 \widetilde A_{i \alpha j \beta} \frac{\p^2 \widetilde u_j}{\p x_\alpha \p x_\beta}\\
& = (\la(x) + \mu(x))\sum_{j = 1}^3 \frac{\p^2 \widetilde u_j}{\p x_i \p x_j} + \mu(x) \sum_{j = 1}^3 \frac{\p^2 \widetilde u_i}{\p x_j^2} +  E_i(t, x, \widetilde u, \nabla \widetilde u, \widetilde f, \nabla \widetilde f),
\end{split}
\eeq
where $E_{i}$ denotes the nonlinear terms. Because $\widetilde A_{i\alpha j \beta}$ comes from a scalar energy function, 
 we know (see e.g.\ \cite[Section 1]{DH}) that $\widetilde A_{i\alpha j \beta} = \widetilde A_{i \beta j \alpha} $ are symmetric. Moreover, because of the assumptions on $\la, \mu$ and the compactness of $\overline\Omega$, $\widetilde A$ satisfy the strong ellipticity condition, namely there exists $\delta > 0$ such that
\beq
\widetilde A_{i \alpha j \beta} \xi_i \xi_j \zeta_\alpha \zeta_\beta \geq \delta |\xi|^2 |\zeta|^2,  \ \  \xi, \zeta \in \mbr^3,
\eeq
for all $\widetilde u$ in a sufficiently small open neighborhood $\mco$ of the zero function in $C^m(\mbr\times \overline\Omega)$ such that $\det (I + \nabla \widetilde u) > 0$. 

Now  $\widetilde u = u - \widetilde f$  satisfies the equation 
\beqq\label{eqnon1}
\begin{gathered}
 \frac{\p^2 \widetilde u}{\p t^2}  -  \nabla\cdot \mca(\widetilde u, \widetilde f) = \mcf, \ \ \text{ in } \mbr\times \overline \Omega,\\
 \widetilde u = 0, \ \ \text{ in } (\mbr_+\times \p\Omega) \cup (\mbr_-\times \overline\Omega),
\end{gathered}
\eeqq
where 
\beq
\mcf = S(\widetilde f) - \frac{\p^2 \widetilde f}{\p t^2} \in C^{m-2}([0, T]\times \overline \Omega).
\eeq 
It is clear that 
\beq
\|\mcf\|_{C^{m-2}} \leq C \|\widetilde f\|_{C^m} \leq C\|f\|_{C^m}. 
\eeq
 Then the assumptions of \cite[Theorem 5.2]{DH} are all satisfied and the problem can be reduced to the following abstract problem studied in \cite[Section 4]{DH}: for any $T_0 > 0,$ consider the initial value problem
\beqq\label{eqabs1}
\begin{gathered}
 \frac{\p^2 u}{\p t^2}  +  E(t, x,  u, \p  u, f, \p  f) u = F, \ \  \text{ in } [0, T_0]\times \overline \Omega,\\
 u = 0,  \ \ \text{ in } \mbr_-\times \overline \Omega,
 \end{gathered}
\eeqq
where $E$ satisfies the assumptions (E1)-(E4) and $F$ (automatically) satisfies assumptions (g1)-(g2) of \cite[Section 4]{DH}. Here, to conform with the notations in \cite{DH}, we have changed the meaning of  $u$ and $f$ so that they are $C^m$ functions on $\mbr\times \overline\Omega.$ Also, we let $H_s = W^{s, 2}(\Omega, \mbr), V = W^{1, 2}_0(\Omega, \mbr)$ and $X_s = V\cap H_s$.  In \cite[Theorem 4.2]{DH}, the local in time existence was established for this abstract problem. In particular, for any $F\in C^{m}([0, T]\times \overline\Omega)$, there exists a $T_0> 0$ and a unique solution 
\beq
u\in \bigcap_{k = 0}^m C^{m-k}([0, T_0]; X_{m-k}). 
\eeq
The proof of Theorem 5.2 of \cite{DH} follows from this result. Here, we claim that 
 for fixed $T_0>0$, if $F$ is sufficiently small, there exists a unique solution $u$ of \eqref{eqabs1} as above and $\|u\|_{C^m} \leq C \|F\|_{C^m}$. 
The proof of the claim is based on modifying that of Theorem 4.2 of \cite{DH}, which is essentially built upon Theorem 4.1 of \cite{DH} for a simplified version of the problem \eqref{eqabs1} i.e.
\beqq\label{eqabs2}
\begin{gathered}
 \frac{\p^2 u}{\p t^2}  +  A(u) u = F, \ \ \text{ in } [0, T_0]\times \overline \Omega,\\
 u = 0,  \ \ \text{ in } \mbr_-\times \overline \Omega,
 \end{gathered}
\eeqq
where $A$ satisfies the assumptions in Section 4 of \cite{DH}. To clearly indicate the modifications we need, we shall prove our claim for this problem.  

For $M, T > 0$, we define a function space $Z(M, T)$ consisting of all functions $w$ satisfying
\beq
w\in \bigcap_{k = 1}^m W^{k, \infty}([0, T]; H_{m-k}),  \ \ \text{ess-sup}_{t\in [0, T]}\sum_{k=0}^3 \|w(t)\|_{m-k}^2 \leq M^2.  
\eeq
For $w\in Z(M, T_0)$, consider the linearized problem
\beqq\label{eqabslin}
\begin{gathered}
 \frac{\p^2 u}{\p t^2}  +  A(t, x, w) u = F, \ \ \text{ in } [0, T_0]\times \Omega,\\
 u = 0,  \ \ \text{ in } \mbr_-\times \overline \Omega.
 \end{gathered}
\eeqq 
For this problem, Theorem 3.1 of \cite{DH} shows that there exists a unique solution $u\in \bigcap_{k = 1}^m W^{k, \infty}([0, T_0];$ $H_{m-k})$ with the estimate
\beq
\sum_{k = 0}^m \|u(t)\|_{m-k}^2 \leq C_0 N(T_0) e^{K_0T_0}, \ \ t\in [0, T_0], 
\eeq
where $C_0, K_0$ are positive constants depending only on the coefficient of the equation, and 
\beq
N(T_0) = \sup_{t\in [0, T_0]}\sum_{k = 0}^{m-2} \|F(t)\|_{m-2-k}^2. 
\eeq
We observe that $N(T_0) = O(\eps^2)$ if $\|F\|_{C^m} \leq \eps.$ 
We denote by $\mct$ the map which maps $w\in Z(M, T)$ to the solution of \eqref{eqabslin}. We let 
\beq
M_0^2 = 4 N(T_0) C_0 e^{K_0 T_0} = O(\eps)
\eeq 
and choose $\eps$ sufficiently small so that $\mct$ maps $Z(M_0, T_0)$ into itself. Following the rest of proof of \cite[Theorem 4.1]{DH} especially equation (4.37), we see that the map is a contraction if 
\beq
CM_0 T_0 e^{CM_0 T_0} < 1.
\eeq
We choose $\eps_0$ sufficiently small so this is true. This finishes the proof of the claim. This implies the claim for system \eqref{eqabs1} which further concludes the proof of the theorem.
\epf

\section{Microlocal analysis of the linearized system}\label{sec-lin}
We consider the initial boundary value problem for the linearized equation \eqref{eqlin1} recalled below
 \beqq\label{eqlin}
 \begin{gathered}
 P u(t, x) = \frac{\p^2 u(t, x)}{\p t^2}  -  \nabla\cdot \tilde S(x, u(t, x))  = 0, \ \ (t, x)\in \mbr\times  \Omega, \\
 u(t, x)   = f(t, x), \ \ (t, x)\in\mbr_+\times \Omega, \\ 
 u(t, x)   = 0,   \ \ (t, x) \in \mbr_-\times \Omega.
\end{gathered}
 \eeqq
where 
 \beq
 \tilde S_{mn}(x, u) = \la(x) \frac{\p u_j}{\p x_j}  \delta_{mn} + \mu(x)  (\frac{\p u_m}{\p x_n} + \frac{\p u_n}{\p x_m}). 
 \eeq
 Our goal is to construct boundary sources $f$ so that the solution $u$ has conormal type singularities propagating into the region $\Omega.$ Such $u$ will be called {\em distorted plane waves}. We start with basic microlocal analysis for boundary value problems of the linear system. 
 
Let $\xi = (\xi_1, \xi_2, \xi_3) \in \mbr^3$ be the dual coordinate of $x$ in $T^*_x\overline\Omega$ and we let $(t, x; \tau, \xi) \in T^*(\mbr\times \overline\Omega)\backslash 0$ be the local coordinates. The Euclidean metric of $\mbr^3$ is used to define inner product $ \xi\cdot \xi, \forall \xi\in \mbr^3$ and to identify tangent and co-tangent vectors on $\mbr^3$. For a non-zero direction $\xi\in \mbr^3\backslash 0$, we denote by $\pi = \pi(\xi) = \xi\otimes\xi/(\xi\cdot\xi)$ the orthogonal projection to $\xi. $ From \cite[Proposition 4.1]{HU}, we know that $P$ is a system of real principal type (in the sense of Dencker \cite{Den}) with principal symbol
\beq
\textbf{p} = \textbf{p}_S(\id - \pi) + \textbf{p}_P \pi,
\eeq
where 
\beq
\textbf{p}_{P/S}(t, x, \tau, \xi) = \tau^2 - \langle \xi, \xi\rangle_{P/S}.
\eeq
For $\mu>0$, we see from \eqref{eqchar} that $0< \langle \xi, \xi \rangle_{S} < \langle \xi, \xi \rangle_{P}, \xi\in \mbr^3\backslash 0$. It is well-known that the system $P$ can be decoupled as follows. We decompose $u$ to the P/S modes as 
 \beq
 \begin{gathered}
 u^P = \Pi_P u = \lap^{-1} \nabla (\nabla\cdot u) \text{ and }  u^S = \Pi_S  u = (\id - \Pi_P) u,
 \end{gathered}
 \eeq
 where $\nabla = (\p_{x_1}, \p_{x_2}, \p_{x_3})$ is the gradient  and $\lap = \sum_{i = 1}^3 \p_{x_i}^2$ is the Laplacian.  
Observe that the symbols of $\Pi_P, \Pi_S$ are $\sigma(\Pi_P)(x, \xi) = \pi(\xi)$ and $\sigma(\Pi_S)(x, \xi) = \id - \pi(\xi), (x, \xi)\in T^*\mbr^3$. 
It follows from Taylor's diagonalization method \cite{Tay} (see also \cite[Lemma 2.1]{SD}) that  $Pu = 0$ is equivalent to 
 \beqq\label{eqlin3}
 \begin{gathered}
 \frac{\p^2 u^P}{\p t^2}  =   [(\la + 2\mu) \lap + B_1]u^P  + R_1 u,   \\
 \frac{\p^2 u^S}{\p t^2}  =  [\mu  \lap + B_2] u^S  + R_2 u,
  \end{gathered}
 \eeqq
 where $B_1, B_2$ are first order pseudo differential operators (denoted by $\Psi^1(\mbr^3)$) and $R_1, R_2$ are smoothing operators. The boundary data $f$ can be decomposed to $f = f^P+ f^S$ so the system \eqref{eqlin1} is decoupled  up to a smoothing term.
 
For the two symbols $\textbf{p}_{P/S}$,  the corresponding Hamiltonian vector fields are 
\beq
H_{\textbf{p}_{P/S}} = -2\tau \frac{\p }{\p t} + \sum_{i = 1}^3 [\frac{\p \langle \xi, \xi \rangle_{P/S}}{\p \xi_i} \frac{\p }{\p x_i} - \frac{\p \langle \xi, \xi \rangle_{P/S}}{\p x_i} \frac{\p }{\p \xi_i}].
\eeq 
The integral curves on $T^*(\mbr\times \overline \Omega)$ are called bicharacteristics.  For $x\in \p \Omega, \xi \in T_x^*\overline\Omega$, we define the projection $\pi_\p: T_x^*(\mbr\times \overline\Omega)\rightarrow T_x^*(\mbr\times \p\Omega)$ by $\pi_\p(\xi) = \xi|_{T^*_x(\p \Omega)}.$ 
The point $\gamma = (t, x; \tau, \pi_{\p}(\xi))\in T^*(\mbr\times \p\Omega)\backslash 0$ is called {\em elliptic, hyperbolic or glancing} for P/S mode if the following quadratic equation in $z$
\beq
\textbf{p}_{P/S}(t, x; \tau, \xi - z\nu(x)) = 0
\eeq
 has no real roots, two distinct real roots or a double real roots, see \cite[Section 4]{HU}. The cotangent bundle $T^*(\mbr\times \p \Omega)$ is decomposed into elliptic regions $\mce_{P/S}$, hyperbolic regions $\mch_{P/S}$ and the glancing hypersurfaces $\mcg_{P/S}$ for the P/S modes. Because of the assumption that $\mu>0$, it is easy to see that $\mce_S\subset\mce_P$ and $\mch_P\subset\mch_S$. We let $\mcg = \mcg_P\cup \mcg_S.$ A simple real root $z$ is called forward (backward) if the bicharacteristic curve starting in direction $\xi - z\nu$ enters $\mbr\times \Omega$ when time increases (decreases). We denote by $z_{P/S}$ the forward real root or the complex root $z$ with positive imaginary part of $\textbf{p}_{P/S}(t, x, \tau, \xi-z\nu) = 0$, and we use $\xi_{P/S} = \xi - z_{P/S} \nu(x)$.  
 
 Consider the displacement-to-traction map of the linear system \eqref{eqlin1}, that is $\La_{lin}(f) = \nu \cdot \tilde S(u)$. We will see later in \eqref{eqdtnlin} that this is just the linearization of the displacement-to-traction map for the nonlinear system \eqref{eqnon1}. It is proved in \cite[Proposition 4.2]{HU} that $\La_{lin}$ is a first order pseudo-differential operator near every non-glancing point $\gamma \in T^*(\mbr\times \p\Omega)\backslash \mcg.$

For a Lagrangian submanifold $\La$ of $T^*M$ e.g. $M = \mbr\times \Omega$, the Lagrangian distributions of order $\mu$ are denoted by $I^\mu(\La)$, see \cite{Ho4} for the definition. Let $K$ be a codimension $k$ submanifold of $M$. The conormal bundle $N^*K = \{(x, \zeta)\in T^*M\backslash 0: x\in K, \zeta|_{T_xK} = 0\}$ is a Lagrangian submanifold. The conormal distributions of order $\mu$ to $K$ are denoted by $I^\mu(N^*K)$. \\
 
Now let $K$ be a codimension one submanifold of $\mbr\times \p\Omega$ (hence codimension two in $\mbr\times \overline\Omega$).
We use  $N^*_\p K$ to denote the conormal bundle of $K$ as a submanifold of the boundary $\mbr\times \p\Omega$ and $N^*K$ the conormal bundle in $\mbr\times \overline\Omega.$ We assume that $N_\p^*K\cap \mch_P$ has an open interior and consider distributions $f\in I^\mu(N_\p^*K)$. Indeed, we are interested in the singularities of $f$  in the hyperbolic directions.  
We introduce 
\beq
\begin{gathered}
\La^P_K = (\mbr\times \overline\Omega)\cap \big(\bigcup_{s\geq 0}\exp sH_{\textbf{p}_P}(N^*K \cap \Sigma_P) \big), \\
 \La^S_K =  (\mbr\times \overline\Omega) \cap \big( \bigcup_{s\geq 0}\exp sH_{\textbf{p}_S}(N^*K \cap \Sigma_S)\big). 
 \end{gathered}
\eeq
These are Lagrangian submanifolds of $T^*(\mbr\times \overline\Omega)$. Their projections to $\mbr\times \overline\Omega$ are geodesic flow out of $N^*K$ with respect to the Lorentzian metrics $-dt^2 + g_{P/S}$. 
 
\begin{prop}\label{proplin1}
Let $K, f$ be defined as above and $u$ be the solution of \eqref{eqlin1} with boundary source $f$. Let $f = f^P + f^S$ and $u = u^P+ u^S$. We have the following conclusions.
\begin{enumerate}
\item There exists (Fourier integral) operators $Q_{bdy}^{P/S}$ such that $u^{P/S} = Q_{bdy}^{P/S}(f^{P/S}) \in I^{\mu - 1/4}(\La_K^{P/S})$ are Lagrangian distributions. 
\item Let $(z, \zeta)\in T^*(\mbr\times \overline\Omega)$ lie on the bicharacteristic strip of $H_{\textbf{p}_{P/S}}$ from $(z_0, \zeta_{0, P/S})$ for some $(z_0, \zeta_0)\in T^*(\mbr\times \p\Omega)$. Then the principal symbols of $u^{P/S}$ and $f^{P/S}$ are related by 
\beq
\sigma(u^{P/S})(z, \zeta) = Q_{bdy}^{P/S}(z, \zeta, z_0, \zeta_0)\sigma(f^{P/S})(z_0, \zeta_0),
\eeq
where $Q_{bdy}^{P/S}$ are $3\times 3$ invertible matrices and $bdy$ stands for boundary value problem. 
\end{enumerate}
\end{prop}

\begin{proof}
For simplicity, we use $Z=\mbr\times \overline\Omega$ and $Y = \mbr\times \p \Omega$. Locally near $Y$, we can make a change of variable to flat the boundary. Then the problem \eqref{eqlin1} is equivalent to the Cauchy problem for the second order system $Pu \in C^\infty(Z)$ with Cauchy data 
\beq
Cu =(\rho_0 u, \rho_0 \nabla\cdot\tilde S(u)) = (f,  \La_{lin} (f)),
\eeq
where $\rho_0$ is the restriction operator to $Y.$  In particular, $\rho_0$ is an Fourier integral operator in $I^{1/4}(Z, Y; R_0)$, where the canonical relation
\beq
R_0 = \{((z_0, \zeta_0), (z, \zeta))\in  (T^*Y\times T^*Z)\backslash 0: z_0 = z, \zeta_0 = \pi_\p(\zeta) = \zeta|_{T_{z_0}^*Y}\},
\eeq
see \cite[Section 5.1]{Du}. According to \cite[Theorem 5.2.1]{Du}, there exist Fourier integral operators $Q_0 \in I^{-1/4}(Z, Y; C_0)$ and $Q_1 \in I^{-1-1/4}(Z, Y; C_0)$ which are maps $\mce'(Y)\rightarrow \mcd'(Z)$ such that
\beq
P Q_i \in C^\infty(Z), \ \ \rho_0 Q_j = \delta_{0j}, \ \ \rho_0 \La Q_j = \delta_{1j},\quad i,j = 0, 1,
\eeq
where $C_0$ is the canonical relation 
\beq
\begin{gathered}
\{((z, \zeta), (z_0, \zeta_0)): (z, \zeta) \in T^*Z \text{ is on the bicharacteristic strip of $\textbf{p}$ through some}\\
\text{$(z_0, \hat \zeta)\in T^*Z$ such that $\pi_\p(\hat \zeta) = \zeta_0,$ for $(z_0, \zeta_0)\in T^*Y$}\}.
\end{gathered}
\eeq 
Suppose that $f\in I^\mu(N^*_\p K)$ is conormal. By the composition of Fourier integral operators (see e.g.\ \cite{Ho4}), we have $ Q_0 f  \in I^{\mu-1/4}(\La_K),  Q_1 f \in I^{\mu-1-1/4}(\La_K)$. So the solution $u = Q_0f + Q_1 f \in I^{\mu-1/4}(\La_K)$. Suppose that $((z, \zeta), (z_0, \zeta_0)) \in C_0$, then the principal symbol 
\beq
\sigma(u)(z, \zeta) = \sigma(Q_0)(z, \zeta, z_0, \zeta_0)\sigma(f)(z_0, \zeta_0),
\eeq
where $Q_0$ is invertible. Finally, we apply these arguments to the decoupled system and let $Q_{bdy}^{P/S} = Q_0^{P/S} + Q_1^{P/S} \in I^{-1/4}(Z, Y; C_0^{P/S})$ where 
\beq
\begin{gathered}
C_0^{P/S} = \{((z, \zeta), (z_0, \zeta_0)): (z, \zeta) \in T^*Z\cap \Sigma_{P/S} \text{ is on the bicharacteristic strip of $\textbf{p}_{P/S}$}\\
\text{through some $(z_0, \hat \zeta)\in T^*Z\cap\Sigma_{P/S}$ such that  $\pi_\p(\hat \zeta) = \zeta_0,$ for $(z_0, \zeta_0)\in T^*Y$}\}.
\end{gathered}
\eeq 
This completes the proof. 
\end{proof}

At last, we use the proposition to construct distorted plane waves. Let $\gamma_0 = (t_0, x_0, \tau_0, \xi_0)\in T^*(\mbr\times \p\Omega)\backslash \mcg, t_0 > 0,  x_0\in \p\Omega$ be a hyperbolic point in $\mch_P\subset \mch_S$. We let $K_0$ be a codimension one submanifold of $\mbr\times \p\Omega$ so that $\gamma_0\in N^*_\p K_0$. For $\delta > 0$ sufficiently small, we define 
\beq
K(\gamma_0; \delta) =  K_0 \cap \{(t, x) \in \mbr\times \p \Omega:  |t-t_0|< \delta, \text{dist}(x, x_0) < \delta\},
\eeq
which is a small neighborhood of $(t_0, x_0)$ contained in $K_0.$ 
Then $\Gamma_0(\delta) \doteq N^*_\p K(\gamma_0; \delta)$ is a small open neighborhood of $\gamma_0$ and $\Gamma_0(\delta)\cap \mch_P\neq\emptyset.$  As $\delta\rightarrow 0$, the set $\Gamma_0(\delta)$ tends to the vector $\gamma_0$. Now we consider their flow out  under the Hamilton vector fields of $\textbf{p}_{P/S}$   
\beq
\begin{gathered}
\La^P(\gamma_0; \delta) = \bigcup_{s\geq 0}\exp sH_{\textbf{p}_P}(N^*K(\gamma_0; \delta)\cap \Sigma_P), \\
 \La^S(\gamma_0; \delta) = \bigcup_{s\geq 0}\exp sH_{\textbf{p}_S}(N^*K(\gamma_0; \delta)\cap \Sigma_S). 
 \end{gathered}
\eeq
which are Lagrangian submanifolds of $T^*(\mbr\times \overline \Omega)$. As $\delta\rightarrow 0$, they tend to the forward bicharacteristics corresponding to $\gamma^{P/S}_0 = (t_0, x_0, \tau_0, \xi_{0, P/S})$. By our non-conjugate point assumption, we know that the projections of $\La_0^{P/S}$ to $\mbr\times \overline \Omega$ should be co-dimension one submanifolds $\mcp_0, \mcs_0.$  So we have
\beq
\La^P(\gamma_0; \delta) = N^*\mcp_0, \ \  \La^S(\gamma_0; \delta) = N^*\mcs_0.
\eeq 
Also, as $\delta \rightarrow 0$, $\mcp_0$ tends to the geodesic of the metric $-dt^2 + g_P$ from $\gamma_0^P$ and $\mcs_0$ tends to the geodesic of the metric $-dt^2 + g_S$ from $\gamma_0^S$. 
For $f \in I^{\mu+1/4}(N_\p^*K(\gamma_0; \delta))$, the solution $u$ of \eqref{eqlin1} satisfies $u = u^P + u^S, u^{P}\in I^\mu(N^*\mcp_0), u^S \in I^\mu(N^*\mcs_0)$, which is called a {\em distorted plane wave}.  See Figure \ref{picdistor}. We see that for $\delta$ small, the singular supports of $u^{P/S}$ are close to the corresponding geodesics from $\gamma_0^{P/S}.$ 
 
\begin{figure}[htbp]
\centering
\includegraphics[scale=0.6]{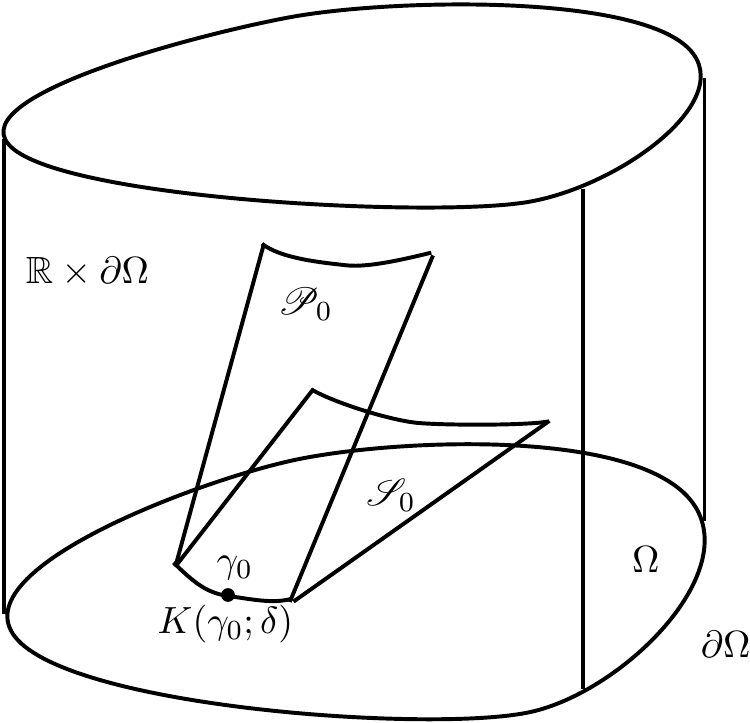}
\caption{Construction of distorted plane waves.  }
\label{picdistor}
\end{figure}
 
\section{The nonlinear interaction}\label{sec-non}
\subsection{Construction of sources}\label{sec-source}
We consider the nonlinear effects in this section. First, we construct two distorted plane waves. 
\begin{definition}\label{def-distor}
Let $\gamma_1, \gamma_2 \in T^*(\mbr_+\times \p \Omega)\backslash \mcg$ be hyperbolic points and construct two sources $\fone \in I^{\mu+ 1/4}(N^*_\p K(\gamma_1; \delta))$ and $\ftwo \in I^{\mu + 1/4}(N^*_\p K(\gamma_2; \delta))$  with $\mu < -23/4$ as in the end of Section \ref{sec-lin}. The corresponding distorted plane waves are denoted by $\uone, \utwo$. We write
\beq
u^{(\bullet)} = u^{(\bullet), S} + u^{(\bullet), P}, \ \ \bullet = 1, 2, 
\eeq
such that
\beq
\begin{gathered}
u^{(\bullet), P} \in I^{\mu}(\La^{P}(\gamma_\bullet; \delta))=  I^{\mu}(N^*\mcp_{\bullet}),\\
u^{(\bullet), S} \in I^{\mu}(\La^{S}(\gamma_\bullet; \delta)) = I^{\mu}(N^*\mcs_\bullet).
\end{gathered}
\eeq
Here, $\mcp_\bullet, \mcs_\bullet$ are codimension one submanifolds of $\mbr\times \overline \Omega.$ We assume that $\mcp_i \cap \mcs_i = \emptyset, i = 1, 2,$ (i.e.\ no self-interactions) and that 
\beq
\begin{gathered}
\mcp_1\cap\mcp_2 = \mcz_{PP}, \quad \mcs_1\cap\mcs_2 = \mcz_{SS}, \quad \mcp_1\cap \mcs_2 = \mcz_{PS}, \quad \mcs_1\cap \mcp_2 = \mcz_{SP},
\end{gathered}
\eeq
where the above intersections are either empty or transversal so the $\mcz_\bullet$ are codimension two submanifolds. 
\end{definition}

We would like to construct a source $f = \eps_1 \fone + \eps_2 \ftwo$ for two small parameters $\eps_1, \eps_2 > 0$ so that the linearized solutions are distorted plane waves. In general, this might lead to reflections of the waves at the boundary and it becomes difficult to determine the nonlinear responses. Therefore, we proceed as follows. 

\begin{prop}\label{propbdata}
For $\fone, \ftwo, \uone, \utwo$ in Def.\ \ref{def-distor} and $\eps_1, \eps_2$ sufficiently small, there exists $f_\eps \in C^2(\mbr \times \p\Omega)$ supported in $\mbr_+\times \p\Omega$ so that the solution $u_\eps$ of \eqref{eqnon} has the expansion
\beqq\label{eqasym}
u = \eps_1 \uone + \eps_2 \utwo + \eps_1^2 u^{(11)} + \eps_2^2 u^{(22)} + \eps_1\eps_2 u^{(12)} + o(\eps_1^2) + o(\eps_2^2),
\eeqq
where $u^{(\bullet)}, \bullet = 11, 12, 22$ are determined by $\uone, \utwo$ through \eqref{eqlin2}. 
\end{prop}

In the expansion \eqref{eqasym}, we let $v = \eps_1\uone + \eps_2\utwo$ and call it the {\em linear response}. The terms $u^{(11)}, u^{(22)}$ and $u^{(12)}$ are called {\em nonlinear responses}. We are particularly interested in $u^{(12)}$ as we shall show below that it contains new singularities which do not belong to the linear response. 
The point of the proposition is revealed in the displacement-to-traction map $\La$. We have
\beqq\label{eqdtnlin}
\p_{\eps_i}\La(f_\eps)|_{\eps_i = 0} =  \nu\cdot \tilde S(u^{(i)})|_{\mbr\times \p\Omega} = \La_{lin}(f^{(i)}), \ \ i = 1, 2,
\eeqq
and
\beqq\label{eqdtnnon}
\p_{\eps_1}\p_{\eps_2}\La(f_\eps)|_{\eps_1=\eps_2=0} = \nu\cdot \tilde S(u^{(12)})|_{\mbr\times \p\Omega} + \nu\cdot (\nabla\cdot\mcg(u^{(1)}, u^{(2)}))|_{\mbr\times \p\Omega},
\eeqq
where $\mcg(\cdot, \cdot)$ is the quadratic term in \eqref{eqlin2}, see also \eqref{eqlin3}. 

\bpf[Proof of Prop.\ \ref{propbdata}]
For $\eps_1, \eps_2> 0$, we take 
\beq
f_{\eps} = \eps_1\uone|_{\mbr\times \p\Omega} + \eps_2\utwo|_{\mbr\times \p\Omega} +  f_{\eps^2},
\eeq
where $\tilde f_{\eps^2}$ consists of higher order terms in $\eps_1, \eps_2$ and is to be specified below. From the finite speed of propagation for the linear system, we see that $f_{\eps} =  \eps_1\fone+\eps_2\ftwo$ modulo higher order terms in a sufficiently small neighborhood of $\gamma_1, \gamma_2$.  Now consider the regularity. Recall that for a codimension $k$ submanifold $K$ of $M$ of dimension $n$, we have
\beqq\label{eqspace}
I^\mu(N^*K) \subset H^s(M) \subset C^r(M),
\eeqq
where $s < -\mu - n/4$ and $r < s-n/2$. We should take $\mu < - 9/2$ so that $\fone, \ftwo \in C^2(\mbr\times \p \Omega)$. We apply Theorem \ref{thmwell}. For any $T_0>0$, there exists $\eps_0> 0$ such that for $\eps_1, \eps_2 < \eps_0$, there exists a unique solution $u_\eps$ of \eqref{eqnon1} with boundary source $f_\eps$ such that 
\beq
u_\eps \in E^2(\mbr\times \overline \Omega) \doteq \bigcap_{k = 0}^2 C^k([0, T_0]; H^{2-k}(\overline \Omega)) \subset H^2(\mbr\times \overline\Omega)
\eeq
and we have the asymptotic expansion \eqref{eqasym} 
in which the remainder terms are also in $E^2(\mbr\times \overline \Omega)$.  Now we need more regularity so that $u_\eps\in C^2(\mbr\times \overline \Omega)$. Thus, we demand that in \eqref{eqspace} $s = 5$ and $\mu< -5- 3/4 = -23/4.$ Then  we let 
\beq
f_{\eps^2} = [\eps_1^2 u^{(11)} + \eps_2^2 u^{(22)} + \eps_1\eps_2 u^{(12)}]|_{\mbr\times \p\Omega} + f_{\eps^3},
\eeq
where $f_{\eps^3} \in C^2(\mbr\times \p\Omega)$ and $f_{\eps^3}= o(\eps_1^2) + o(\eps_2^2).$ We see that $f_{\eps^3}$ will not affect the terms in the asymptotic expansion \eqref{eqasym}. This finishes the proof.
\epf

We remark that since we will only concern $\gamma_1, \gamma_2$ so the corresponding  bicharacteristics do not meet at the boundary, the wave front of $\nu\cdot (\nabla\cdot\mcg(u^{(1)}, u^{(2)}))|_{\mbr\times \p\Omega}$ in \eqref{eqdtnnon}  will be contained in that of $\uone$ and $\utwo$. Thus it suffices to find the singularities of $\nu\cdot \tilde S(u^{(12)})|_{\mbr\times \p\Omega}$ in $u^{(12)}$ which we   do next.

\subsection{Generation of the nonlinear response} 
Among all the nonlinear terms in \eqref{eqnon}, we only consider the quadratic terms in $S(u)$, denoted by $G(u, u)$ where
 \beqq\label{eqG}
 \begin{split}
 G_{mn}(u, w)& =    \la  \tilde e_{jj}  \frac{\p w_m}{\p x_n} + \ha \la (\frac{\p u_k}{\p x_j})  (\frac{\p w_k}{\p x_j})  \delta_{mn} + 2\mu  \tilde e_{nj} \frac{\p w_m}{\p x_j} + \mu \frac{\p u_k}{\p x_m} \frac{\p w_k}{\p x_n} \\
 &+ A\tilde e_{mj}\tilde f_{nj}  +B(2\tilde e_{jj}\tilde f_{mn} + \tilde e_{ij}\tilde f_{ij}\delta_{mn}) + C  \tilde e_{ii}\tilde f_{jj} \delta_{mn},
 \end{split}
 \eeqq
 where $\tilde f_{mn} = \ha (\dfrac{\p w_m}{\p x_n} + \dfrac{\p w_n}{\p x_m})$. Because $G(u, w)$ is not symmetric, we let 
\beq
\mcg(u, w) = G(u, w) + G(w, u). 
\eeq
Then we see that for $v = \eps_1\uone + \eps_2\utwo$, 
\beq
G(v, v) = \eps_1^2 \ha \mcg(\uone, \uone) +  \eps_2^2 \ha \mcg(\utwo, \utwo) +  \eps_1\eps_2\mcg(\uone, \utwo).
\eeq
Thus $u^{(12)}$ is the solution of
\beqq\label{eqlin3}
\begin{split}
 Pu^{(12)} =  \frac{\p^2 u^{(12)}}{\p t^2}  - \nabla\cdot \tilde S(u^{(12)}) &= \nabla \cdot \mcg(\uone, \utwo)
\end{split}
 \eeqq 
 with zero initial and boundary conditions. This is the precise form of the equation \eqref{eqlin2}. If we choose the parameter $\delta$ in the distorted plane waves sufficiently small,  $\mcg(\uone, \utwo)$ is compactly supported in $\mbr_+\times \overline\Omega$. Thus by the finite speed of propagation, we can treat \eqref{eqlin3} as a source problem on $\mbr\times \mbr^3$ before the waves reaches the boundary. Although this is not necessary for our proof,  it is worth mentioning that in \cite{Ra0}  Rachele showed the determination of $\la, \mu$ and their normal derivatives to any order on the boundary $\mbr\times \p \Omega$ from $\La_{lin}$. Thus one can ignore the boundary and extend the the system \eqref{eqlin3} to $\mbr\times \mbr^3$. Because of the P/S decomposition, we  have
 \beq
\begin{split}
 \mcg(u^{(1)}, u^{(2)}) &= \mcg(u^{(1), P}, u^{(2), P})+ \mcg(u^{(1), P}, u^{(2), S})+ \mcg(u^{(1), S}, u^{(2), P}) + \mcg(u^{(1), S}, u^{(2), S})\\
 & = \mcg^{PP} + \mcg^{PS} + \mcg^{SP}+ \mcg^{SS},
 \end{split}
 \eeq  
where the $\mcg^\bullet$ in the second line corresponds to the four terms in the first line. These terms  represent the P--P interactions, P--S interactions, S--P interactions and S--S interactions. Their singularities can be described using the notion of paired Lagrangian distributions.   Let $M$ be an $n$-dimensional smooth manifold. For two Lagrangians $\La_0, \La_1 \subset T^*M$ intersecting cleanly at a co-dimension $k$ submanifold i.e. $T_q\La_0\cap T_q\La_1  = T_q(\La_0\cap \La_1),\ \ \forall q\in \La_0\cap \La_1,$ the paired Lagrangian distribution associated with $(\La_0, \La_1)$ is denoted by $I^{p, l}(\La_0, \La_1)$. The wave front sets of such distributions are contained in $\La_0\cup \La_1.$ We refer the reader to \cite{MU, GrU93} for the precise definition and properties. 

Now consider $\mcg_{PP}$ and assume $\mcz_{PP}\neq\emptyset$. From \cite[Lemma 4.1]{WZ}, we know that the components of $\nabla u^{(\bullet), P}$ are in $I^{\mu+1}(\La^P_\bullet), \bullet = 1, 2$. Then we can apply \cite[Lemma 2.1]{GrU93} to get
\beq
\mcg_{PP} \in I^{\mu+1, \mu+2}(N^*\mcz_{PP}, N^*\mcp_1) + I^{\mu+1, \mu+2}(N^*\mcz_{PP}, N^*\mcp_2).
\eeq
Using \cite[Lemma 4.1]{WZ} again, we get 
 \beq
\nabla\cdot \mcg_{PP} \in I^{\mu+2, \mu+2}(N^*\mcz_{PP}, N^*\mcp_1) + I^{\mu+2, \mu+2}(N^*\mcz_{PP}, N^*\mcp_2).
\eeq
 The wave front set of $\mcg^{PP}$ is contained in the union of $N^*\mcz_{PP}$ and $N^*\mcp_1, N^*\mcp_2$. For the propagation of the nonlinear response, we are interested in the co-vectors  of $N^*\mcz_{PP}$ which are also in $\Sigma_P$ or $\Sigma_S$. 
 \begin{lemma}\label{lmint1}
Suppose that $\mcp_1$ intersect $\mcp_2$ transversally at $\mcz_{PP} \neq \emptyset$. Then 
\begin{enumerate}
\item $(N^*\mcz_{PP} \backslash (N^*\mcp_1\cup N^*\mcp_2))\cap \Sigma_P  = \emptyset$. 
\item For any $p\in \mcz_{PP}$, there are two linearly independent vectors $\zeta_+, \zeta_-\in \Sigma_S \cap N^*\mcz_{PP}$ at $p$. 
\end{enumerate}
 \end{lemma}
 \bpf 
 We remark that (1) is a known fact, but we give an elementary proof below for completeness. Let $p = (t, x) \in \mcz_{PP}$ and $\zeta^{(1)}\in N_p^*\mcp_1, \zeta^{(2)}\in N_p^*\mcp_2$. We write $\zeta^{(i)} = (\tau^{i}, \xi^{i}), \tau^{i}\in \mbr, \xi^{i}\in \mbr^3, i = 1, 2$. Then we have
\beq
\begin{gathered}
 (\tau^i)^2 = (\la + 2\mu) |\xi^i|^2, \ \ i = 1, 2.
\end{gathered}
\eeq
Now consider vectors $\zeta = a\zeta^{(1)} + b\zeta^{(2)} \in N^*\mcz_{PP}, a, b\in\mbr$. Without loss of generality, we assume that $a \neq 0$ and rescale the vectors so  that $|\xi^{(\bullet)}|=1$ and $a = 1$. If $\zeta\in \Sigma_P$, we have
\beq
\begin{gathered}
 ( \tau^1 + b\tau^2)^2 = (\la + 2\mu) | \xi^1 + b\xi^2|^2 
\Rightarrow b(1-   \xi^1\cdot \xi^2) = 0.
\end{gathered}
\eeq  
Because $\xi^1\cdot\xi^2 \neq 0$, we conclude that $b = 0$ which implies $\zeta = \zeta^{(1)}$. 
If $\zeta \in \Sigma_S$, we must have
\beq
\begin{gathered}
 ( \tau^1 + b\tau^2)^2 = \mu | \xi^1 + b\xi^2|^2  \\
\Rightarrow (\la + \mu)b^2 + 2((2\mu +\la) - \mu \xi^1\cdot \xi^2)) b + (\la+ \mu) = 0.
\end{gathered}
\eeq  
Because we assumed $\la+ \mu > 0$, the equation above is quadratic and the determinant is positive if $\xi^1\cdot \xi^2\neq 1$, which is automatically true by the transversal intersection assumption. In this case, we get two real distinct roots $b_+, b_-$ and two co-vectors in $\Sigma_S\cap N^*\mcz_{PP}$
 \beq
 \zeta^+ = \zeta^{(1)} + b_+ \zeta^{(2)}, \ \ \zeta^- = \zeta^{(1)} + b_- \zeta^{(2)}.
 \eeq
\epf

Similarly, we have 
\begin{lemma}\label{lmint2}
Assume that $\mcp_1$ intersects $\mcs_2$ transversally at $\mcz_{PS} \neq \emptyset$. Then 
\beq
\nabla\cdot \mcg_{PS} \in I^{\mu+2, \mu+2}(N^*\mcz_{PS}, N^*\mcp_1) + I^{\mu+2, \mu+2}(N^*\mcz_{PS}, N^*\mcs_2).
\eeq
For any $p\in \mcz_{PS}$, there exists a unique $\zeta_+\in \Sigma_P\backslash  N^*\mcp_1$ and $\zeta_-\in \Sigma_S\backslash N^*\mcs_2$ at $p$. The same conclusion holds for $\mcg_{SP}$.
\end{lemma}
\bpf 
Let $(t, x) \in \mcz_{PS}$ and $\zeta^{(1)}\in N^*\mcp_1, \zeta^{(2)}\in N^*\mcs_2$. We write $\zeta^{(i)} = (\tau^{i}, \xi^{i}), \xi^{i}\in \mbr^3, |\xi^i|=1, i = 1, 2$. Then we have
\beq
  (\tau^1)^2 = (\la + 2\mu) |\xi^1|^2, \ \ (\tau^2)^2 = \mu |\xi^2|^2.
\eeq
Now consider vectors $\zeta = \zeta^{(1)} + b\zeta^{(2)} \in N^*\mcz_{PS},  b\in\mbr$. If $\zeta \in \Sigma_P$, we must have
\beq
\begin{gathered}
 ( \tau^1 + b\tau^2)^2 = (\la +2\mu) | \xi^1 + b\xi^2|^2 \\
 \Rightarrow b^2 (\la+\mu) + 2b (\xi^1 \cdot \xi^2(\la+ 2\mu) -\sqrt{\mu(\la+2\mu)}) = 0.
\end{gathered}
\eeq  
The equation has two real solutions. One is $b = 0$ corresponding to the P vector $\zeta^{(1)}$ and $b_P\neq0$ corresponding to a new vector in $\Sigma_P$. Now consider the vector $\zeta$ in $\Sigma_S.$ We arrive at the equation
 \beq
\begin{gathered}
 ( \tau^1 + b\tau^2)^2 =  \mu | \xi^1 + b\xi^2|^2  \\
\Rightarrow 2b(\sqrt{(\la + 2\mu)\mu} - \mu\xi^1\cdot \xi^2) + (\la + \mu) = 0.
\end{gathered}
\eeq  
So we get one non-trivial solution $b_S$. Thus, we conclude that $N^*\mcz_{PS}$ has one P vector and one S vector. Similar conclusion holds for $\mcg^{SP}.$
\epf

Finally, we have
\begin{lemma}\label{lmint3}
Assume that $\mcs_1$ intersects $\mcs_2$ transversally at $\mcz_{SS} \neq \emptyset$. Then 
\beq
\nabla\cdot \mcg_{SS} \in I^{\mu+2, \mu+2}(N^*\mcz_{SS}, N^*\mcs_1) + I^{\mu+2, \mu+2}(N^*\mcz_{SS}, N^*\mcs_2), \text{ and }
\eeq
\begin{enumerate}
\item $(N^*\mcz_{SS}\backslash (N^*\mcs_1\cup N^*\mcs_2))\cap \Sigma_S = \emptyset$.
\item For $p\in \mcz_{SS}$, there are two linearly independent vectors $\zeta_+, \zeta_-\in N^*\mcz_{SS} \cap \Sigma_P$  if the following interaction condition holds: 
\beqq
\begin{gathered}
\text{for  $\zeta^i = (\tau^i, \xi^i) \in N_p^*\mcs_i, i = 1, 2$, we have  $\cos(\xi^1, \xi^2) < \frac{-\la}{\la + 2\mu}$}. \tag{I}
\end{gathered}
\eeqq
\end{enumerate}
\end{lemma}
\bpf
Let $(t, x) \in \mcz_{SS}$ and $\zeta^{(1)}\in N^*\mcs_1, \zeta^{(2)}\in N^*\mcs_2$. We write $\zeta^{(i)} = (\tau^{i}, \xi^{i}), \xi^{i}\in \mbr^3, i = 1, 2$ so that $|\xi^i| = 1.$ We have $
 (\tau^i)^2 = \mu |\xi^i|^2 = \mu, \ \ i = 1, 2. $
Now consider vectors $\zeta = \zeta^{(1)} + b\zeta^{(2)} \in N^*\mcs_{12},  b\in\mbr$. If $\zeta \in \Sigma_S$, we must have
\beqq\label{eqtemp1}
\begin{gathered}
 ( \tau^1 + b\tau^2)^2 = (\la + 2\mu) | \xi^1 + b\xi^2|^2 \\
\Rightarrow (\la + \mu)b^2 + 2( (\la + 2\mu)\xi^1\cdot \xi^2 - \mu) b + (\la + \mu) = 0.
\end{gathered}
\eeqq
The equation has two distinct real roots $b_\pm$ if  
\beq 
\xi^1\cdot \xi^2 < \frac{ -\la}{\la + 2\mu}. 
\eeq 
In this case, we get two P vectors in $N^*\mcz_{SS}$.
\epf
We remark that by our assumption $\la+\mu > 0, \mu >0$, we have $-\la/(\la+2\mu)\in (-1, 0)$. Thus one can find $\zeta^1, \zeta^2$ at $p\in \mcz_{SS}$ so that the interaction condition holds. \\
 
Next, let's recall the microlocal parametrix for  $Pu=f$ on $\mbr\times \mbr^3$. Let $\diag = \{ (z, z') \in  \mbr^4\times \mbr^4: z = z'\}$ be the diagonal of the product space  and $N^*\diag$ be the conormal bundle minus the zero section. We regard the symbols $\textbf{p}_{P/S}(z, \zeta) $ as functions on the product space. Then we denote by $\La^P, \La^S$ the flow out of $N^*\diag$ under $H_{\textbf{p}_P}, H_{\textbf{p}_S}$. So $\La^{P/S}$ are Lagrangian submanifolds of $T^*(\mbr^4\times \mbr^4)$. We know that the system $P$ is decomposed to the diagonal form. So according to \cite{MU}, there exits a distribution 
\beq
\begin{gathered}
Q_{sour} = Q_{sour}^P + Q_{sour}^S, \\
 Q_{sour}^P\in I^{-\frac 32, -\frac 12}(N^*\diag, \La^P),\ \ Q_{sour}^S\in I^{-\frac 32, -\frac 12}(N^*\diag, \La^S)
 \end{gathered}
\eeq
such that $PQ_{sour} = \id$ up to a smoothing term.  Here, the subscript $sour$ stands for the source problem. In the following, we denote 
\beq
\begin{gathered}
\La^{PPS} = \La^S\circ N^*\mcz_{PP}, \ \ \La^{SSP} = \La^P\circ N^*\mcz_{SS},\\
\La^{PSP} = \La^P\circ N^*\mcz_{PS}, \ \ \La^{PSS} = \La^{S}\circ N^*\mcz_{PS},\\
\La^{SPP} = \La^P\circ N^*\mcz_{SP}, \ \ \La^{SPS} = \La^{S}\circ N^*\mcz_{SP}.
\end{gathered}
\eeq
Again, because of the no conjugate point assumption, these are conormal bundles. In fact, $\La^{\bullet} = N^*\mcz_{\bullet}, \bullet = PPS, SSP, PSP, PSS, SPP, SPS$ where $\mcz_{\bullet}$ are codimension one submanifolds of $\mbr\times \overline\Omega.$ 
 
\begin{theorem}\label{thmresp}
Suppose $\uone, \utwo$ are distorted plane waves in Def.\ \ref{def-distor}.  
\begin{enumerate}
\item The solution to \eqref{eqlin2} can be decomposed as 
\beq
u^{(12)} = u^{PPS} + u^{PSP} + u^{PSS} + u^{SPP} + u^{SPS} + u^{SSP},
\eeq
such that microlocally away from $\La_1^P\cup\La_2^P\cup\La_1^S\cup\La_2^S$, we have
\beq
\begin{gathered}
u^\bullet \in I^{2\mu + \frac 52}(\La^\bullet), \quad \bullet = PPS, PSP, PSS, SPP, SPS, SSP.
\end{gathered}
\eeq
Moreover, $u^{SSP}$ is smooth on $\La^{SSP}$ unless $\mcs_1, \mcs_2$ satisfy the interaction condition. 

\item If $\mcz_{\bullet}$ intersect $\mbr\times \p\Omega$ transversally at $\mcy_{\bullet}$, then 
\beq
\p_{\eps_1}\p_{\eps_2}\La(f_\eps)\in I^{2\mu+\frac{9}{4}}(N_\p^*\mcy_{\bullet})
\eeq
are conormal distributions.
\item Consider the symbol at $\mcy_\bullet = \mcy_{PPS}. $ Let $(z_0, \zeta_0)\in T^*Z$ and the bicharacteristic from $(z_0, \zeta_0)$ intersect $T^*Y$ transversally. Let $(z, \zeta) =  \La^P(z_0, \zeta_0)$ and $(z_|, \zeta_|) =  R_0(z, \zeta)$ with $R_0$ the canonical relation of the restriction operator. Then the principal symbol satisfies
\beq
\sigma(\p_{\eps_1}\p_{\eps_2}\La(f_\eps))(z_|, \zeta_|) = \sigma(\rho_0)(z_|, \zeta_|, z, \zeta)Q^S_{sour}(z, \zeta, z_0, \zeta_0)\sigma(\nabla\cdot\mcg^{PP})(z_0, \zeta_0),
\eeq
where $Q^S_{sour}$ and $\sigma(\rho_0)$ are $3\times 3$ invertible matrices. Similar statements hold for $\mcy_{\bullet}, \bullet = PSP, PSS, SPP, SPS, SSP.$
\end{enumerate}
\end{theorem}
\bpf
We analyze $u^{PPS}$ and the others are similar. We know that away from $\La^P_1, \La^P_2$, $\nabla\cdot\mcg^{PP}\in I^{2\mu-2}(N^*\mcz_{PP})$. 
Because $N^*\mcz_{PP}$ intersect $\Sigma_S$ transversally, we can apply Proposition 2.2 of \cite{GrU93} to get
\beq
u^{PPS} = Q^S(\nabla\cdot \mcg^{PP}) \in I^{\mu+2 + \mu+2 -\frac 32, -\frac 12}(N^*\mcz_{PP}, \La^{PPS})
\eeq
modulo a distribution whose wave front set is contained in a neighborhood of $\La^P_1, \La^P_2$. Thus, away from $\mcz_{PP}$, $u^{PPS} \in  I^{2\mu + \frac 52}(\La^{PPS})$. 

Next, if $\La^{PPS}$ intersect the boundary $Y$ transversally, we see that $\p_{\eps_1}\p_{\eps_2}\La(f_\eps) = \rho_0 (u^{PPS})$ near the intersection. By the composition of FIOs, we know the term is a conormal distribution with order $-1/4$ less than that of $u^{PPS}$. 
\epf

We remark that because $\mcz_{\bullet}$ are of codimension one, the singularities of the nonlinear response $u^{\bullet}$ above are of the same type as a distorted plane wave, see Figure \ref{figinter}. Also, if $\p \Omega$ is strictly convex with respect to $g_{P/S},$ the intersection of $\mcz_{\bullet}$ and $\mbr\times \p\Omega$ is transversal. We also remark that $u^{PPS}$ can be regarded as consisting of two waves in view of Lemma \ref{lmint1}. The same is true for $u^{SSP}$ in view of Lemma \ref{lmint3}.

\begin{figure}[htbp]
\centering
\includegraphics[scale=0.7]{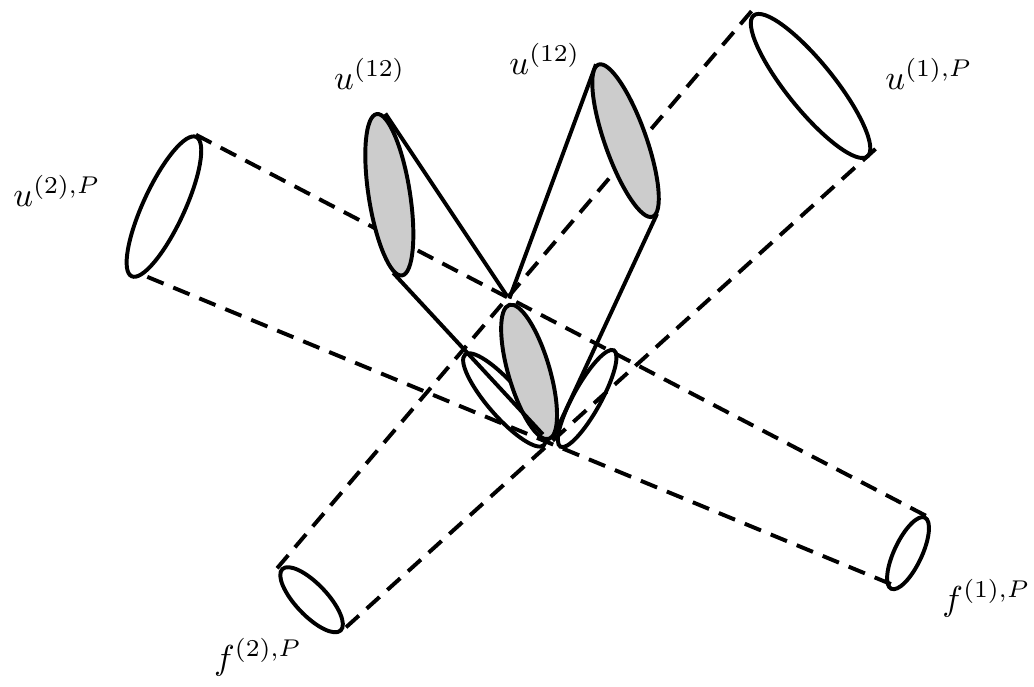}
\caption{Illustration of the interaction of two P waves. The picture is in $\mbr^3$. The white ellipses show the evolution of the singular supports of two P waves for different time $t> 0$ along the two paths. The gray ellipses show the generation and evolution of the wave fronts of the generated S wave.}
\label{figinter}
\end{figure}

\subsection{Symbols of the nonlinear responses}\label{sec-symbol}
We determine the symbol of the interaction terms  and show that they are not always vanishing. This would confirm the generation of new waves. Roughly, there are three kinds of interactions so we split the section to three subsections.

\subsubsection{P--P interactions} We take $u^{(\bullet), P}\in I^\mu(\La^P_\bullet),  \bullet = 1, 2, $ and consider the singularities of $\mcg^{PP}$.  For ease of calculation, we introduce some quantities for the interaction.  Let $z\in \mcp_1\cap\mcp_2$ and $(z, \zeta^{1}) \in \La_1^P, (z, \zeta^{2}) \in \La_2^P$. Assume that $\zeta = \zeta^{1} + \zeta^{2} \in \Sigma_S$. Let $\zeta^{\bullet} = (\tau^\bullet, \xi^\bullet), \xi^\bullet\in \mbr^3, \bullet = 1, 2.$ We call the plane determined by $\xi^1, \xi^2$ the {\em interaction plane.} Then $\xi = \xi^1 + \xi^2$. Because we consider the S wave, we let $\xi^H$ be a unit vector in the interaction plane perpendicular to $\xi$ and $\xi^V$ be a unit vector orthogonal to the interaction plane.  We define the angles $\alpha, \psi$ through
 \begin{gather*}
 \xi^1\cdot \xi^2 = |\xi^1||\xi^2|\cos \alpha,\ \
 \xi\cdot \xi^2 = |\xi||\xi^2|\cos\psi. 
 \end{gather*}
See Figure \ref{picplane}. The angles $\alpha, \psi$ and the relations have been used in the literatures (e.g.\ \cite{KSC}) and they are physically useful.

\begin{figure}[htbp]
\centering
\includegraphics[scale=0.7]{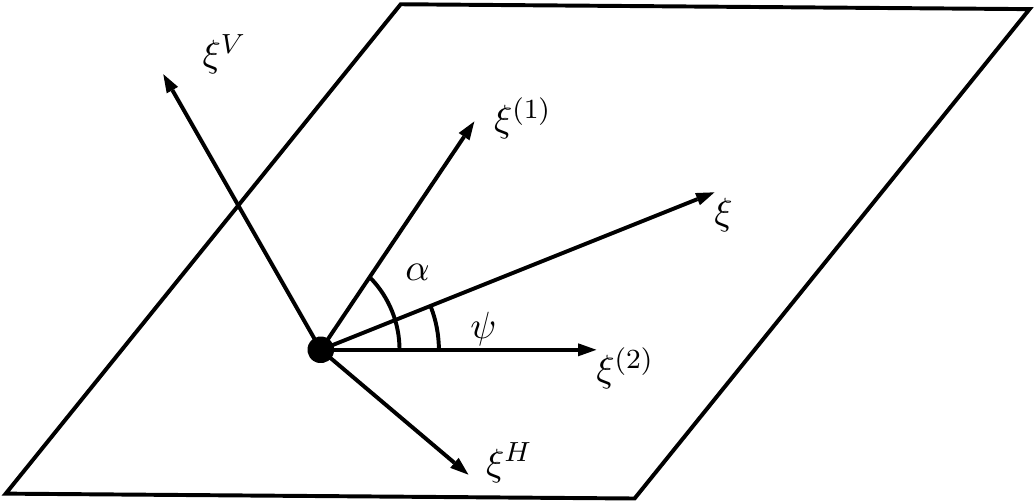}
\caption{The interaction plane for P--P wave interactions}
\label{picplane}
\end{figure}

We consider the term $\mcg(u^{(1), P}, u^{(2), P})$. The symbol of $u^{(\bullet), P}$ at $(z, \zeta^\bullet) \in \La_\bullet^P$ is the projection of $\sigma(u^{(\bullet)})(z, \zeta^{\bullet})$ by $\sigma(\Pi_P)$ at $z$ along the $\xi^\bullet, \bullet = 1, 2$ direction.  Thus we can write $\sigma(u^{(\bullet), P})(z, \zeta^\bullet) = a^\bullet \xi^\bullet$ for some constant $a^\bullet$.  Then consider 
  \beq
 \tilde e^{\bullet}_{mn} =  \ha (\frac{\p u^{(\bullet), P}_m}{\p x_n} + \frac{\p u^{(\bullet), P}_n}{\p x_m}) \in I^{\mu+1}(\La^P_\bullet), \ \ \bullet = 1, 2.
 \eeq
The corresponding symbol of  $\tilde e^\bullet$ is 
\beq
U^\bullet =  \imath a_\bullet \xi^{\bullet, T}\xi^\bullet, \ \ \bullet = 1, 2.
\eeq
Here, $\imath^2 = -1$ and $\xi^\bullet$ are regarded as row vectors hence $\xi^{\bullet, T}\xi^\bullet$ are  $3\times 3$ symmetric matrices. Let $\zeta = \zeta^1 + \zeta^2$. We denote the principal symbol of $\mcg(u^{(1), P}, u^{(2), P})\in I^{\mu+1}(N^*\mcz_{PP})$ at $(z, \zeta)\in N^*\mcz_{PP}\cap \Sigma_S$ by $\textbf{g}(z, \zeta)$.  We recall the symbol calculation from \cite[Lemma 3.3]{LUW1}. For $u^{(\bullet), P} \in I^\mu(N^*\mcp_\bullet)$, consider the principal symbol of $u^{(1), P} u^{(2), P} \in I^{2\mu + 3}(N^*\mcz_{PP}\backslash N^*\mcp_1\cup N^*\mcp_2)$ in local coordinates of $p \in \mcz_{PP}$. For $\zeta = \zeta_1 + \zeta_2 \in N^*\mcz_{PP}$ with $\zeta_\bullet \in N^*\mcp_\bullet$, we have 
\beq
\sigma(u^{(1), P}u^{(2), P})(z, \zeta) = \sigma(u^{(1), P})(z, \zeta_1)\sigma(u^{(2), P})(z, \zeta_2). 
\eeq
Here, we absorbed the $(2\pi)^{-1}$ factor in \cite[Lemma 3.3]{LUW1} to the symbols. Then we use \cite[Lemma 4.1]{WZ} and the expression \eqref{eqG} of $G(u, u)$ to get 
\beqq\label{eqG}
\begin{split}
 -\textbf{g}_{mn}  
 = &  \la  U^1_{jj}  U^2_{mn}+ \la  U^2_{jj}  U^1_{mn} +  \la U^1_{kj}U^2_{kj} \delta_{mn} \\
 & +  2 \mu U^1_{nj} U^2_{mj} + \mu U^1_{km}U^2_{kn} +2 \mu  U^2_{nj} U^1_{mj} + \mu U^2_{km}U^1_{kn}\\
&+ A [U^1_{mj}U^2_{nj} + U^2_{mj}U^1_{nj}] + B(2U^1_{jj}U^2_{mn} + 2U^2_{jj}U^1_{mn} + 2U^1_{ij}U^2_{ij}\delta_{mn}) + C 2U^1_{ii}U^2_{jj} \delta_{mn}\\
=& (\la + 2B) a_1 a_2 [|\xi^1|^2 \xi^{2}_m\xi^2_n + |\xi^2|^2 \xi^{1}_m\xi^1_n + (\xi^1\cdot \xi^2)^2 \delta_{mn}]\\
 &+(A +3 \mu) a_1 a_2 [\xi^{1}_m \xi^1_k \xi^{2}_k \xi^2_n +  \xi^{2}_m\xi^2_k \xi^{1}_k\xi^1_n  ]  
 + 2 C a_1a_2 |\xi^1|^2|\xi^2|^2 \delta_{mn}.
 \end{split}
\eeqq
(The negative sign is due to the symbol of two derivatives.) Then we get 
 \beq
\textbf{h}(z, \zeta) = \sigma(\nabla\cdot \mcg(u^{(1), P}, u^{(2), P}))(z, \zeta) = \imath \textbf{g}(z, \zeta) \xi.
 \eeq
 Because we consider the S wave propagation, we project the symbol $\textbf{h}$ along the $\xi^H$ and $\xi^V$ directions, which are denoted by $\textbf{h}^{SH}, \textbf{h}^{SV}$ respectively. Then the symbol $\textbf{h}^{S\bullet}$ are 
 \beqq\label{eqxi}
 |\xi^\bullet|^{-2} (\xi^\bullet \textbf{g}(z, \zeta) \xi )\xi^\bullet = (\xi_m^\bullet \textbf{g}_{mn}\xi_n)\xi^\bullet, \ \ \bullet = V, H.
 \eeqq
We first compute the symbol $\textbf{h}^{SV}$: 
\beq
\begin{gathered}
 \imath \textbf{h}^{SV}(z, \zeta) = (\la + 2B)a_1a_2 [|\xi^1|^2 (\xi^{V}\cdot \xi^{2})(\xi^2\cdot \xi) + |\xi^2|^2 (\xi^{V}\cdot \xi^{1})(\xi^1\cdot \xi)]\xi^V\\
+ (3\mu + A)a_1a_2[ (\xi^{V}\cdot \xi^{1})(\xi^1 \cdot \xi^{2})(\xi^2\cdot \xi) + (\xi^{V}\cdot \xi^{2})(\xi^2 \cdot \xi^{2})(\xi^1\cdot \xi)]\xi^V = 0,
\end{gathered}
\eeq
because of $\xi^V$ is perpendicular to the interaction plane. Next we consider the symbol $\textbf{h}^{SH}$:
\beq
\begin{split}
\imath \textbf{h}^{SH}(z, \zeta)  &= (\la+ 2B) a_1 a_2 [|\xi^1|^2 ((\xi^H \cdot \xi^{2})(\xi^2 \cdot \xi) + |\xi^2|^2 (\xi^H \cdot \xi^{1})(\xi^1\cdot \xi) + (\xi^1\cdot \xi^2)^2 (\xi^H\cdot \xi)]\xi^H\\
 & + (3\mu +A) a_1a_2 [(\xi^H \cdot \xi^{1})(\xi^1\cdot \xi^{2})(\xi^2\cdot \xi) + (\xi^H \cdot \xi^{2})(\xi^2 \cdot \xi^{1})(\xi^1\cdot \xi) ] \xi^H\\
 &+ 2 C a_1a_2 |\xi^1|^2|\xi^2|^2 (\xi^H\cdot \xi) \xi^H \\
 & = a_1 a_2|\xi^1|^2|\xi^2|^2 |\xi| [ (\la + 2B) (\sin\psi \cos\psi - \sin(\alpha-\psi) \cos (\alpha - \psi)  )\\
 &+ (A + 3\mu)( -\sin(\alpha-\psi)\cos \alpha \cos \psi + \sin\psi\cos\alpha\cos(\alpha-\psi))] \xi^H.
\end{split}
\eeq
Using trigonometry identities, we obtain that
 \beq
  \textbf{h}^{SH}(z, \zeta)    = -\imath a_1 a_2|\xi^1|^2|\xi^2|^2 |\xi|  (\la + 3\mu + A+ 2B)\cos\alpha \sin(2\psi-\alpha) \xi^H.
\eeq
If $\la + 3\mu + A + 2B \neq 0$, this is non-vanishing for $(\psi, \alpha)$ in any open set of $(0, \pi)^2$. In this sense, we call the symbol {\em generically non-vanishing}. We also observe that in the principal symbol of $\mcg^{PP}$ the information of $C(x)$ is lost.

\subsubsection{P--S interactions} Consider the term $\mcg^{PS} = \mcg(u^{(1), P}, u^{(2), S}).$ The analysis for $\mcg^{SP}$ is the same. For simplicity, we let $\uone = u^{(1), P}$ and $\utwo = u^{(2), S}$. For the principal symbol of $u^{(2), S}$ at $(z, \zeta) = (t, x; \tau, \xi)\in \La^S$, we observe that
\beq
\sigma(\tilde e^{(2)}_{ii})(z, \zeta) =  \sum_{i=1}^3  \xi_i \frac{|\xi^2|\delta_{il} - \xi_i\xi_l}{|\xi|^2}\imath \sigma(u^{(2), S}_l)  = 0.
\eeq
(Another way to see this is that the S component of $u$ is divergence free.) This type of term appears in $C  \tilde e_{ii}\tilde e_{jj} \delta_{mn}$ of $G(u, u)$ so $C(x)$ does not appear in the symbols for interactions involving S waves. Therefore, before we compute the symbols explicitly, we  proved
\begin{prop}\label{lmC}
For the two distorted plane waves $\uone, \utwo$ in Def.\ \ref{def-distor}, the principal symbols of the corresponding terms $\mcg^{\bullet}, \bullet = PP, PS, SP, SS$  are independent of $C(x)$. So are the symbols of the nonlinear responses $u^\bullet, \bullet = PPS, PSP, PSS, SPS, SPP, SSP.$
\end{prop} 

\begin{figure}[htbp]
\centering
\includegraphics[scale=0.7]{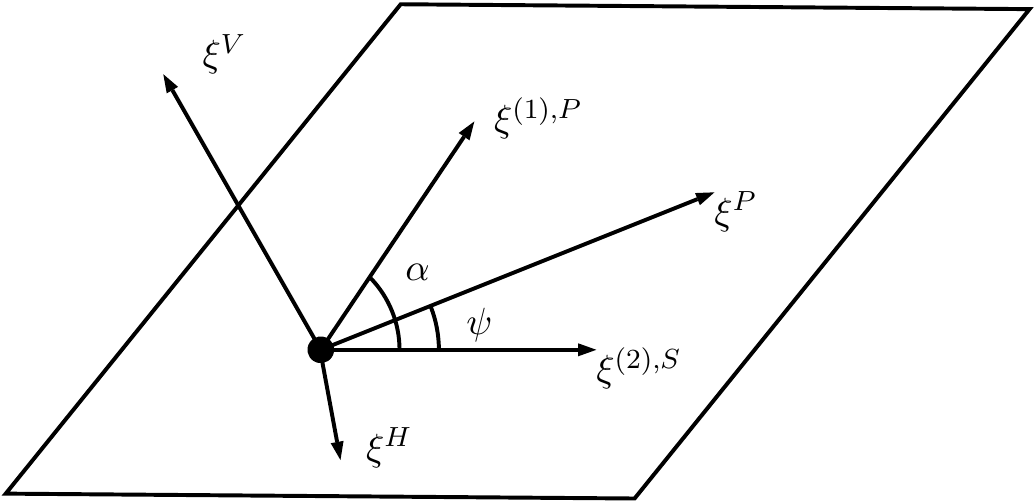}
\includegraphics[scale=0.7]{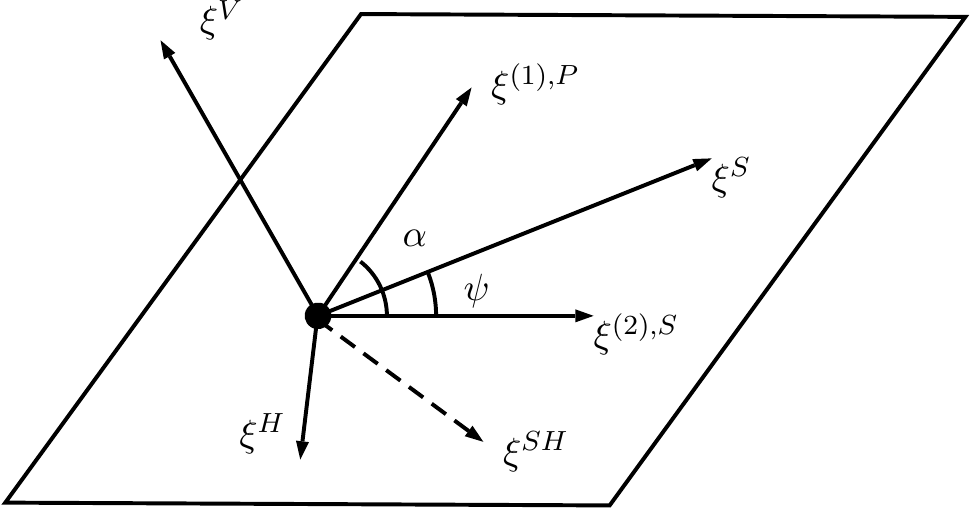}
\caption{The interaction plane for P--S wave interactions. Left: picture for the P mode. Right: picture for the S mode.}
\label{picplaneps}
\end{figure}

Now we proceed to determine the principal symbol of $\mcg^{PS}$. We again introduce the interaction plane to simplify the calculation, see Figure \ref{picplaneps}.  Let $z\in \mcp_1\cap\mcs_2$ and $(z, \zeta^{1}) \in \La_1^P, (z, \zeta^{2}) \in \La_2^S$. Let $\zeta^i = (\tau^i, \xi^i), i = 1, 2$ as before. We call the plane determined by $\xi^1, \xi^2$ the interaction plane. Let $\xi^H$ be a unit vector in the interaction plane orthogonal to $\xi^2$. Then let $\xi^V$ be a unit vector orthogonal to the interaction plane.  

 We first consider the P mode of $\mcg^{PS}$. Assume that $\zeta^P = \zeta^{1} + \zeta^{2} \in \Sigma_P$ and let $\zeta = (\tau^P, \xi^P)$. We define the angles $\alpha, \psi$ through
 \begin{gather*}
\xi^1\cdot \xi^2 = |\xi^1||\xi^2|\cos \alpha,\ \ \xi\cdot \xi^P = |\xi||\xi^P|\cos\psi,
 \end{gather*}
 see the left figure of Figure \ref{picplaneps}. 
Now we can express the principal symbols of $\uone, \utwo$ in terms of these quantities. We let $\sigma(u^{(1), P})(z, \zeta^1) = a \xi^{1}$ for some constant $a$ and we decompose $\sigma(u^{(2), S})(z, \zeta^2) = b_H \xi^{H} + b_V \xi^{V}$ for some constants $b_\bullet$. We let $U^1 = a \xi^{1, T}\xi^1$ so that the principal symbol of $\tilde e^1$ is $\imath U^1.$ Next we let 
\beq
W_{mn} =   \xi^2_n (b_H \xi^H_m + b_V \xi^V_m) = \xi^2_n b_H \xi^H_m +  \xi^2_n  b_V \xi^V_m = W^H_{mn} + W^V_{mn}
\eeq
corresponding to the H, V decomposition. 
We observe that the principal symbol $\sigma(\dfrac{\p u^{(2), S}_m}{\p x_n}) $ is $\imath W_{mn}.$ Now we define
\beq
\begin{gathered}
U^2 = \ha (W + W^T) = \ha\xi^{2, T}(b_H \xi^{H} + b_V \xi^{V}) + \ha(b_H \xi^{H, T} + b_V \xi^{V, T})\xi^2\\
 = \ha b_H (\xi^{2, T} \xi^{H} + \xi^{H, T}\xi^2) + \ha b_V (\xi^{2, T} \xi^{V} + \xi^{V, T}\xi^2) = U^H + U^V,
\end{gathered}
\eeq
where $U^H, U^V$ are defined by the second line. So the principal symbol of $\tilde e^{(2)}$ is $\imath U^2.$ We remark that the $U$ matrices are symmetric but $W$ matrices are not. Because of the H, V decomposition, we will write $\sigma(\mcg^{PS})(z, \zeta) = {\bf g}^H + {\bf g}^V$ 
where ${\bf g}^\bullet, \bullet = H, V$ are defined as
\beqq\label{eqgps}
\begin{split}
-\textbf{g}^\bullet_{mn} &=   \la  U^1_{jj}  W^\bullet_{mn} +  \la U^1_{kj}W^\bullet_{kj} \delta_{mn}  +  2 \mu U^1_{nj} W^\bullet_{mj} + \mu U^1_{km}W^\bullet_{kn} +2 \mu U^\bullet_{nj} U^1_{mj} + \mu W^\bullet_{km}U^1_{kn}\\
&+ A [U^1_{mj}U^\bullet_{nj} + U^\bullet_{mj}U^1_{nj}] + B(2U^1_{jj}U^\bullet_{mn}  + 2U^1_{ij}U^\bullet_{ij}\delta_{mn})\\
 &= \la a b_\bullet (|\xi_1|^2 \xi^\bullet_m \xi^2_n + \xi^1_k \xi^1_j \xi^\bullet_k \xi^2_j\delta_{mn})\\
 & + 2\mu ab_\bullet \xi^1_n\xi^1_j\xi^\bullet_m\xi^2_j + \mu ab_\bullet \xi_k^1\xi^1_m\xi^\bullet_k \xi^2_n + \mu ab_\bullet (\xi^\bullet_n \xi^2_j + \xi^\bullet_j \xi^2_n) \xi^1_m\xi^1_j + \mu ab_\bullet \xi^\bullet_k \xi^2_m\xi^1_k\xi^1_n\\
 &+ A ab_\bullet \ha [\xi^1_m\xi^1_j (\xi^\bullet_n \xi^2_j + \xi^\bullet_j \xi^2_n) + (\xi^\bullet_m \xi^2_j + \xi^\bullet_j \xi^2_m)\xi^1_n\xi^1_j]\\
 & + B ab_\bullet (|\xi_1|^2 (\xi^\bullet_m \xi^2_n + \xi^\bullet_n \xi^2_m) + \xi^1_i\xi^1_j(\xi^\bullet_i \xi^2_j + \xi^\bullet_j \xi^2_i)\delta_{mn}).
\end{split}
\eeqq
 Then we get 
 \beq
\textbf{h}(z, \zeta) = \sigma(\nabla\cdot \mcg(u^{1, P}, u^{2, S}))(z, \zeta) =  \imath (\textbf{g}^H(z, \zeta) + \textbf{g}^V(z, \zeta)) \xi^P. 
 \eeq
 Finally, we project the symbol to $\xi^P$ direction to get the symbol of the P mode:
  \beq 
{\bf h}^{\bullet P} = |\xi^P|^{-2} (\xi^P \textbf{g}^\bullet(z, \zeta) \xi^P )\xi^P, \ \ \bullet =  H, V.
 \eeq 
 We compute 
\beq
\begin{split}
 \imath \textbf{h}^{HP}(z, \zeta)  & = |\xi^P|^{-2} a b_H  (\la + 2B)  [|\xi^1|^2(\xi^P \cdot \xi^2)(\xi^P\cdot \xi^H)  + (\xi^1\cdot \xi^2)(\xi^1\cdot \xi^H)|\xi^P|^2] \xi^P \\
 & +|\xi^P|^{-2} a b_H (3\mu +A)   [ (\xi^P\cdot \xi^2) (\xi^1\cdot \xi^P)(\xi^1\cdot \xi^H) + (\xi^P \cdot \xi^H) (\xi^1\cdot \xi^P)(\xi^1\cdot \xi^2) ]\xi^P \\
 & = a b_H |\xi^1|^2|\xi^2|  |\xi^H| [ (\la + 2B) (-\cos \psi\sin \psi  - \cos\alpha \sin\alpha) \\
 &+  (A + 3\mu)( -\cos \psi \cos (\alpha - \psi) \sin \alpha - \sin\psi \cos(\alpha - \psi)\cos\alpha )]\xi^P.
\end{split}
\eeq
We observe that 
\beq
\begin{gathered}
-\cos \psi \cos (\alpha - \psi) \sin \alpha - \sin\psi \cos(\alpha - \psi)\cos\alpha = -\cos(\alpha - \psi)\sin(\alpha + \psi)\\
 = -\cos \psi\sin \psi  - \cos\alpha \sin\alpha.
 \end{gathered}
\eeq
Thus we have 
\beq
 \textbf{h}^{HP}(z, \zeta)   =   \imath a b_H |\xi^1|^2|\xi^2|   (\la + 2B + A + 3\mu) \cos(\alpha - \psi)\sin(\alpha + \psi) \xi^P.
\eeq
This term is generically non-vanishing when $\la + 2B + A +3\mu \neq 0.$

Next, consider the interactions with the $V$ components of $u^{(2), S}.$
\beq
\begin{split}
 \imath \textbf{h}^{VP}(z, \zeta)  &= |\xi^P|^{-2} a b_V (\la+ 2B)  [|\xi^1|^2(\xi^P \cdot \xi^2)(\xi^P\cdot \xi^V)  + (\xi^1\cdot \xi^2)(\xi^1\cdot \xi^V)|\xi^P|^2] \xi^P \\
 & +|\xi^P|^{-2} a b_V (3\mu +A)   [ (\xi^P\cdot \xi^2) (\xi^1\cdot \xi^P)(\xi^1\cdot \xi^V) + (\xi^P \cdot \xi^V) (\xi^1\cdot \xi^P)(\xi^1\cdot \xi^2) ]\xi^P =0.
\end{split}
\eeq
Thus we conclude that the symbol of $u^{PSP}$ at $(z, \zeta)$ is given by $\textbf{h}^{HP}$  and the term is generically non-vanishing. \\

It remains to consider the generation of S mode from the P--S interaction. In this case, we let $\xi = \xi^S$ and $\xi^{SH}$ be the unit vector in the interaction plane orthogonal to $\xi^S$. We decompose $u$ to the plane determined by $\xi^{SH}$ and $\xi^V$, see the right picture of Figure \ref{picplaneps}. The computation of ${\bf g}$ is the same as \eqref{eqgps} and we have the symbol of $\mcg^{PS}$
 \beq
\textbf{h}(z, \zeta) = \sigma(\nabla\cdot \mcg(u^{(1), P}, u^{(2), S}))(z, \zeta) =  \imath (\textbf{g}^H(z, \zeta) + \textbf{g}^V(z, \zeta)) \xi^S. 
 \eeq
We project the symbol to $\xi^\ast, \ast = V, SH$ directions to get the symbol of the S mode:
  \beq 
{\bf h}^{\bullet \ast} = |\xi^\ast|^{-2} \imath (\xi^\ast \textbf{g}^\bullet (z, \zeta) \xi^S )\xi^\ast, \ \ \bullet = V, H, \quad \ast = V, SH.
 \eeq  
 We compute
 \beq
\begin{split}
 \imath \textbf{h}^{\bullet \ast}(z, \zeta)    & =  a b_\bullet  \la  [|\xi^1|^2   (\xi^\ast \cdot \xi^\bullet)(\xi^2 \cdot \xi^S) + (\xi^1 \cdot \xi^\bullet)(\xi^1 \cdot \xi^2) (\xi^\ast \cdot \xi^S)] \xi^\ast\\
 & + 2\mu ab_\bullet (\xi^\ast\cdot \xi^\bullet)(\xi^1\cdot \xi^2)(\xi^1\cdot \xi^S) \xi^\ast + \mu ab_\bullet (\xi^1\cdot \xi^\ast)(\xi^1\cdot \xi^\bullet)(\xi^2\cdot \xi^S) \xi^\ast\\
 & + \mu a b_\bullet [(\xi^1\cdot \xi^\ast)(\xi^1\cdot \xi^2)(\xi^\bullet \cdot \xi^S) + (\xi^1\cdot \xi^\ast)(\xi^1\cdot\xi^\bullet)(\xi^2\cdot \xi^S) + (\xi^2\cdot \xi^\ast)(\xi^1\cdot \xi^\bullet)(\xi^1\cdot \xi^S)]\xi^\ast \\
 & + Aa b_\bullet \ha [ (\xi^1\cdot \xi^\ast)(\xi^1\cdot \xi^2)(\xi^\bullet \cdot \xi^S) + (\xi^1\cdot \xi^\ast)(\xi^1\cdot \xi^\bullet)(\xi^2 \cdot \xi^S) + (\xi^\bullet\cdot \xi^\ast)(\xi^1\cdot \xi^2)(\xi^1 \cdot \xi^S)\\
 & + (\xi^2\cdot \xi^\ast)(\xi^\bullet\cdot \xi^1)(\xi^1 \cdot \xi^S)  ]\xi^\ast\\
 &+ B ab_\bullet  [|\xi^1|^2   (\xi^\ast \cdot \xi^\bullet)(\xi^2 \cdot \xi^S) + |\xi^1|^2   (\xi^\bullet \cdot \xi^S)(\xi^2 \cdot \xi^\ast)  + 2 (\xi^1\cdot \xi^\bullet)(\xi^1\cdot \xi^2) (\xi^\ast \cdot \xi^S)] \xi^\ast.
\end{split}
\eeq 
We observe that if $\bullet = V, \ast = SH$ or $\bullet = H, \ast = V$ then the term must be zero. So it suffices to consider two cases. 
After some calculations, we find that 
 \beq
\begin{split}
 & \imath \textbf{h}^{H SH}(z, \zeta) \\
  & =  ab_{H} |\xi^1|^2 |\xi^2| |\xi^S| [\la   \cos^2 \psi + \mu (\cos(2\psi) + \cos^2\psi) 
 + \ha A \cos(2\psi) + B(\cos^2\psi - \sin^2\psi)]\xi^{SH} \\
 & =  ab_{H} |\xi^1|^2 |\xi^2| |\xi^S| [(\la +  2\mu + B + \ha A)\cos^2\psi - (\mu + B + \ha A)\sin^2\psi ]\xi^{SH}.
\end{split}
\eeq 
The following lemma is straightforward. 
\begin{lemma}
For $ \psi \in (0, \pi)$,  consider the vector $\vec v(\psi) = [\cos^2\psi, \sin^2 \psi]$. Then $\det(\vec v(\psi_1), \vec v(\psi_2) )$ 
is non-vanishing for $ \psi_1, \psi_2 $ in any open subset of $(0, \pi)^2.$
\end{lemma}
In this sense, we say that the symbol $\textbf{h}^{HSH}$ is generically nonvanishing if $\la +  2\mu + \ha A+ B\neq 0$ or $\mu + \ha A+ B\neq 0$.

Next, we calculate that 
 \beq
\begin{split}
 \imath \textbf{h}^{VV}(z, \zeta)  & =   a b_V    |\xi^1|^2 |\xi^2||\xi^S| [\la \cos\psi + 2\mu \cos\alpha \cos(\alpha - \psi)  + \ha A  \cos\alpha \cos(\alpha - \psi) + B\cos \psi ] \xi^V\\
 & = a b_V    |\xi^1|^2 |\xi^2||\xi^S| [(\la + B) \cos\psi + (2\mu + \ha A)\cos\alpha \cos(\alpha - \psi) ] \xi^V.
\end{split}
\eeq 
Similarly, we conclude that the term is generically non-vanishing if $\la + B\neq 0$ or $2\mu + \ha A\neq 0.$

\subsubsection{S--S interactions}
We let $u^{(1)} = u^{(1), S}$ and $u^{(2)} = u^{(2), S}$ and we consider $\mcg^{SS}$. We decompose the S modes according to the interaction plane, see Figure \ref{picplaness}.  Let $z\in \mcs_1\cap\mcs_2$ and $(z, \zeta^{1}) \in \La_1^S, (z, \zeta^{2}) \in \La_2^S$. Let $\zeta^i = (\tau^i, \xi^i), i = 1, 2$ as before. We call the plane determined by $\xi^1, \xi^2$ the interaction plane. Let $\xi^{i, H}$ be a unit vector in the interaction plane orthogonal to $\xi^i, i = 1, 2.$ Then let $\xi^V$ be a unit vector orthogonal to the interaction plane. We can decompose $u^{(1)}, u^{(2)}$ to H, V modes. 

We decompose $\sigma(u^{(i), S})(z, \zeta^2) = b^i_H \xi^{i, H} + b^i_V \xi^{i, V}$ for some $b^i_\bullet$ constants, $i = 1, 2, \bullet = H, V$.  Similar to the previous case, we let 
\beq
W^i_{mn} =  \xi^i_n (b^i_H \xi^{i, H}_m + b^i_V \xi^{i, V}_m) = \xi^i_n b^i_H \xi^{i, H}_m +   \xi^i_n  b^i_V \xi^{i, V}_m = W^{i,H}_{mn} + W^{i, V}_{mn}, \ \ i = 1, 2,
\eeq
corresponding to the H, V decomposition. The principal symbol $\sigma(\dfrac{\p u^{(i), S}_m}{\p x_n}) $ is $\imath W^i_{mn}.$ Now we define
\beq
\begin{gathered}
U^i = \ha (W^i + W^{i, T}) = \ha\xi^{i, T}(b^i_H \xi^{i, H} + b^i_V \xi^{i, V}) + \ha(b^i_H \xi^{i, H, T} + b^i_V \xi^{i, V, T})\xi^i\\
 = \ha b^i_H (\xi^{i, T} \xi^{i, H} + \xi^{i, H, T}\xi^i) + \ha b^i_V (\xi^{i, T} \xi^{i, V} + \xi^{i, V, T}\xi^i) = U^{i, H} + U^{i, V},
\end{gathered}
\eeq
So the principal symbol of $\tilde e^{(i)}$ is $\imath U^i.$  Because of the H, V decomposition, we will write 
\[ 
\sigma(\mcg^{PS})(z, \zeta) = {\bf g}^{HH} + {\bf g}^{HV} + {\bf g}^{VH} + {\bf g}^{VV},
\]
where ${\bf g}^{\ast \bullet}, \ast, \bullet = \text{H, V}$ are defined as
\beqq\label{eqgps}
\begin{split}
 -\textbf{g}^{\ast \bullet}_{mn}  &=   \la W^{1, \ast}_{kj}W^{2, \bullet}_{kj} \delta_{mn}  +  2 \mu U^{1, \ast}_{nj} W^{2, \bullet}_{mj} + \mu W^{1, \ast}_{km}W^{2, \bullet}_{kn} +2 \mu U^{2, \bullet}_{nj} W^{1, \ast}_{mj} + \mu W^{2, \bullet}_{km}W^{1, \ast}_{kn}\\
&+ A [U^{1, \ast}_{mj}U^{2, \bullet}_{nj} + U^{2, \bullet}_{mj}U^{1, \ast}_{nj}] + 2B U^{1, \ast}_{ij}U^{2, \bullet}_{ij}\delta_{mn},\\
& = \la \xi^{1, \ast}_k \xi^{1}_j \xi^{2, \bullet}_k \xi^2_j \delta_{mn} + \mu (\xi^{1, \ast}_n \xi^{1}_j + \xi^1_n \xi^{1, \ast}_j)\xi^{2, \bullet}_m \xi^2_j + \mu \xi^{1, \ast}_k\xi^1_m \xi^{2, \bullet}_k \xi^2_n \\
& + \mu(\xi^{2, \bullet}_n\xi^2_j + \xi^2_n \xi^{2, \bullet}_j)\xi^{1, \ast}_m\xi^1_j + \mu \xi^{2, \bullet}_k\xi^2_m\xi^{1, \ast}_k\xi^1_n\\
& + \frac 14 A [(\xi^{1, \ast}_m \xi^1_j + \xi^1_m \xi^{1, \ast}_j)(\xi^{2, \bullet}_n\xi^2_j + \xi^2_n\xi^{2, \bullet}_j) + (\xi^{2, \bullet}_m \xi^2_j + \xi^2_m \xi^{2, \bullet}_j)(\xi^{1, \ast}_n\xi^1_j + \xi^1_n\xi^{1, \ast}_j) ]\\
& + \ha B (\xi^{1, \ast}_i \xi^1_j + \xi^1_i \xi^{1, \ast}_j)(\xi^{2, \bullet}_i\xi^2_j + \xi^2_i\xi^{2, \bullet}_j) \delta_{mn}.
\end{split}
\eeqq
These terms represents the interaction of all possible combinations of the H, V modes. 

\begin{figure}[htbp]
\centering
\includegraphics[scale=0.7]{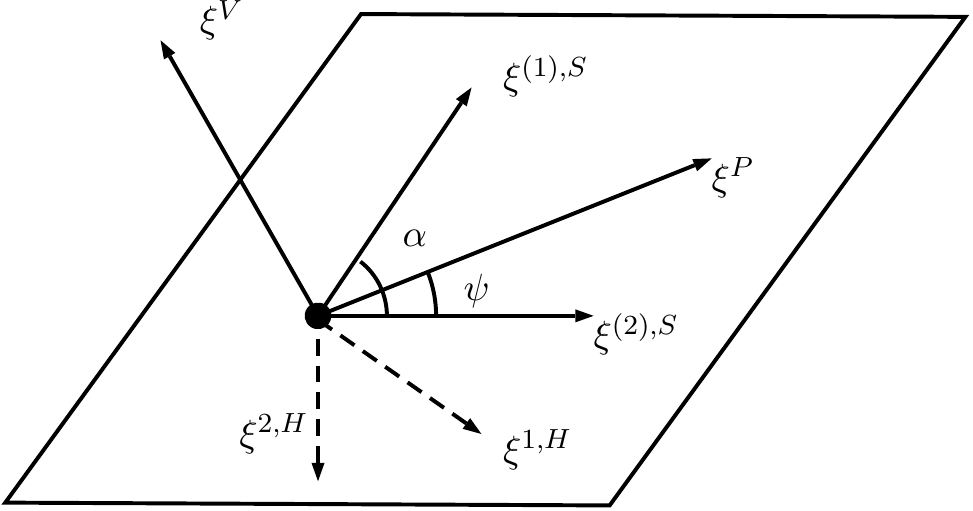}
\caption{The interaction plane for S--S wave interactions.}
\label{picplaness}
\end{figure}
 
 Remember that we are computing the P mode of $\mcg^{SS}$ when the interaction condition is satisfied. So we let $\zeta^P = \zeta^1 + \zeta^2 \in \Sigma^P$ and $\zeta^P = (\tau^P, \xi^P)$. 
As before, we get 
 \beq
 \begin{split}
\textbf{h}(z, \zeta) &= \sigma(\nabla\cdot \mcg(u^{1, S}, u^{2, S}))(z, \zeta) \\
&=  \imath (\textbf{g}^{HH}(z, \zeta) + \textbf{g}^{HV}(z, \zeta) + \textbf{g}^{VH}(z, \zeta) + \textbf{g}^{VV}(z, \zeta)) \xi^P. 
 \end{split}
 \eeq
 Finally, we project the symbol to $\xi^P$ direction to get the symbol of the P mode:
  \beq 
{\bf h}^{\bullet} = |\xi^P|^{-2} (\xi^P \imath \textbf{g}^\bullet(z, \zeta) \xi^P )\xi^P, \ \ \bullet =  HH, HV, VH, VV.
 \eeq 
Because of the orthogonality, one can check that ${\bf h}^{HV} = {\bf h}^{VH} = 0$ (details are omitted).  We compute ${\bf h}^{HH}, {\bf h}^{VV}$ carefully. We have  
 \beq
 \begin{split}
\imath  {\bf h}^{VV} & =  b_V^1b_V^2 |\xi^P|^{-2} [\la |\xi^P|^2 (\xi^1\cdot \xi^2)(\xi^{V}\cdot \xi^{V}) +  2\mu (\xi^1\cdot \xi^P) (\xi^V\cdot \xi^V) (\xi^2\cdot \xi^P) \\
&+ \frac 12 A(\xi^1\cdot \xi^P)(\xi^V\cdot \xi^V)(\xi^2\cdot \xi^P)  + B|\xi^P|^2 (\xi^V\cdot\xi^V)(\xi^1\cdot \xi^2)]\xi^P\\
& =  b_V^1b_V^2|\xi^1||\xi^2| [(\la + B)\cos \alpha +  (\ha A + 2\mu)\cos\psi\cos(\alpha-\psi)]\xi^P.
 \end{split}
 \eeq
 This term is generically non-vanishing if $\la + B\neq 0$ or $\ha A+2\mu\neq 0$. At last, we compute
 \beq
 \begin{split}
\imath |\xi^P|^2 {\bf h}^{HH}   & =  \la b^1_H  b^2_H (\xi^{1, H}\cdot \xi^{2, H})(\xi^1\cdot \xi^2) (\xi^P\cdot \xi^P) \xi^P\\
&+ \mu b^1_Hb^2_H [(\xi^{2, H}\cdot \xi^P)(\xi^1\cdot \xi^2)(\xi^{1, H}\cdot \xi^P) + (\xi^{2, H}\cdot \xi^P)(\xi^2\cdot \xi^{1, H})(\xi^{1}\cdot \xi^P)\\
& + (\xi^{1}\cdot \xi^P)(\xi^2\cdot \xi^P)(\xi^{1, H}\cdot \xi^{2, H}) + (\xi^{1, H}\cdot \xi^P)(\xi^{2, H}\cdot \xi^P)(\xi^{1}\cdot \xi^2) \\
&+ (\xi^{1, H}\cdot \xi^P)(\xi^1\cdot \xi^{2, H})(\xi^{2}\cdot \xi^P) + (\xi^{1, H}\cdot \xi^{2, H})(\xi^2\cdot \xi^P)(\xi^{1}\cdot \xi^P)] \xi^P\\
& + \frac 14 A b^1_Hb^2_H[ (\xi^{1, H}\cdot \xi^P)(\xi^1\cdot \xi^2)(\xi^{2, H}\cdot \xi^{P})+  (\xi^{1, H}\cdot \xi^P)(\xi^1\cdot \xi^{2, H})(\xi^{2}\cdot \xi^{P})\\
& +  (\xi^{1}\cdot \xi^P)(\xi^{1, H}\cdot \xi^2)(\xi^{2, H}\cdot \xi^{P})+  (\xi^{1}\cdot \xi^P)(\xi^{1, H}\cdot \xi^{2, H})(\xi^{2}\cdot \xi^{P})]\xi^P\\
& + \frac 14 A b^1_Hb^2_H[ (\xi^{2, H}\cdot \xi^P)(\xi^1\cdot \xi^2)(\xi^{1, H}\cdot \xi^{P})+  (\xi^{2, H}\cdot \xi^P)(\xi^2\cdot \xi^{1, H})(\xi^{1}\cdot \xi^{P})\\
& +  (\xi^{2}\cdot \xi^P)(\xi^{2, H}\cdot \xi^1)(\xi^{1, H}\cdot \xi^{P})+  (\xi^{2}\cdot \xi^P)(\xi^{2, H}\cdot \xi^{1, H})(\xi^{1}\cdot \xi^{P})]\xi^P\\ 
& + B b^1_H b^2_H [(\xi^{1, H}\cdot \xi^{2, H})(\xi^1\cdot \xi^2) + (\xi^{1, H}\cdot \xi^2)(\xi^1\cdot \xi^{2, H})]\xi^P.
 \end{split}
 \eeq
 Then we have 
 \beq
 \begin{split}
 \imath {\bf h}^{HH} 
  = &b^1_H b^2_H |\xi^1||\xi^2| [\la \cos^2\alpha + \mu(2\cos^2 \alpha - \sin^2\alpha) + \ha A(\cos^2\alpha - \sin^2\alpha) + B(\cos^2\alpha - \sin^2\alpha)] \xi^P \\
 = & b^1_H b^2_H |\xi^1||\xi^2|  [(\la + 2\mu + B + \ha A) \cos^2\alpha - (\mu + B + \ha A)\sin^2\alpha]\xi^P.
 \end{split}
 \eeq
This term is generically non-vanishing if $\la + 2\mu + \ha A+ B\neq 0$ or $\mu+\ha A+ B\neq 0.$\\

 To conclude this section,  we summarize all the possible interactions in Table \ref{tabinter}.
 \begin{table}[htbp]
 \begin{center}
  \begin{tabular}{c | c | c | c}
    \hline
    \hline 
 Interactions & Non-vanishing conditions \\ \hline
   P $+$ P $\rightarrow$ SH & $\la + 2B + 3\mu + A\neq 0$ \\ [3pt]\hline 
    P $+$ SH $\rightarrow$ P  & $\la + 2B + 3\mu + A\neq 0$\\ [3pt]\hline
  P $+$ SH $\rightarrow$ SH & $\la +  2\mu + \ha A + B \neq 0$ or $\mu + \ha A + B\neq 0$\\ [3pt]\hline
        P $+$ SV $\rightarrow$ SV  & $\la + B\neq 0$ or $2\mu + \ha A\neq 0$\\ [3pt]\hline
    SH $+$ SH $\rightarrow$ P & Interaction condition and \\[3pt]
    & $\la +  2\mu + \ha A + B \neq 0$ or $\mu + \ha A + B\neq 0$\\ [3pt]\hline
    SH $+$ SV $\rightarrow \emptyset$  & None \\ [3pt]\hline
    SV $+$ SV $\rightarrow$ P & Interaction condition and \\[3pt]
 &   $\la +  B \neq 0$ or $ 2\mu + \ha A\neq 0$ \\[3pt] \hline
    \hline
  \end{tabular}
\end{center}
\caption{All possible nonlinear interactions. SH stands for the S mode within the interaction plane, SV stands for the S mode perpendicular to the interaction plane. Non-vanishing condition means the principal symbols of the nonlinear responses $u^\bullet$ in Theorem \ref{thmresp} are non-vanishing. The interaction condition is in Lemma \ref{lmint3}.} 
  \label{tabinter}
\end{table}

\section{The inverse problem}\label{sec-inv}
We complete the proofs of Theorem \ref{main} in this section. 
\bpf[Proof of Theorem \ref{main}]
First of all, from the displacement-to-traction map $\La$, we derive the linearized map $\La_{lin}$ which corresponds to the linearized elastic equation \eqref{eqlin1}, see \eqref{eqdtnlin}. This problem have been studied in \cite{Ra} for $\mu > 0, 3\la + 2\mu > 0$ on $\overline \Omega$ and in a more general setting in \cite{HU} for $\mu > 0, \la + \mu > 0$. We conclude that one can determine the P and S wave speed $\sqrt{\la + 2\mu}$ and $\sqrt{\mu}$, hence $\la$ and $\mu$ from $\La_{lin, T_0}, T_0> \text{diam}_S(\overline\Omega).$ 

It is convenient to consider the P, SV wave interaction. So  for any $(t_0, x_0)\in \mbr\times  \Omega$ and $\xi^1, \xi^2$ two linearly independent vectors at $x_0$, we choose two geodesics $c_1(s), c_2(s)$ for $-dt^2 + g_{P}, -dt^2 + g_{S}$ respectively  such that 
\beq
\begin{gathered}
c_\bullet(0) = (t_\bullet, x_\bullet)\in \mbr\times \p\Omega, \ \ c_\bullet(s_\bullet) = (t_0, x_0), \ \ s_\bullet > 0,  \\
\text{ and }\dot c_\bullet(s_\bullet) = (\tau^\bullet, \xi^\bullet), \ \ \bullet = 1, 2.
\end{gathered}
\eeq  
We let $\gamma^\bullet$ be the corresponding cotangent vectors at $(t_\bullet, x_\bullet), \bullet = 1, 2$. Then we construct two distorted plane waves $u^{(1)}, u^{(2)}$ as in Def.\ \ref{def-distor} for $\gamma^1, \gamma^2$ and a small parameter $\delta$. We take $f^{(1)} = f^{(1), P},  f^{(2)} = f^{(2), S}$ hence $u^{(1)} = u^{(1), P}, u^{(2)} = u^{(2), S}$ by Proposition \ref{proplin1}. Following the nonlinear analysis in Section \ref{sec-non}, we see that for $\delta$ sufficiently small, the nonlinear response $u^{PSS}$ in Theorem \ref{thmresp} is a conormal distribution to $\La^{PSS}$ (away from the wave front sets of the linear responses). From the principal symbol of $\p_{\eps_1}\p_{\eps_2}\La(f_\eps)|_{\eps_1=\eps_2=0}$ at the  boundary (for a measurement time $T_0 > 2\text{diam}_S(\Omega)$), we can determine the principal symbols of $u^{PSS}$ at $\mcz_{PS}$. From the symbol of the SV mode, we obtain the value of  $ \la +  B$ and $2\mu + \ha A$ at $x_0$, see Table \ref{tabinter}. This determines the value of $A, B$ at $x_0 $ and finishes the proof of Theorem \ref{main}. 
\epf

\section*{Acknowledgment} 
Maarten V. de Hoop acknowledges and sincerely thanks the Simons Foundation under the MATH$+$X program for financial support. He was                  
also partially supported by the members of the Geo-Mathematical Imaging Group at Rice University. 
Gunther Uhlmann is partially supported by NSF and a Si-Yuan Professorship at HKUST.



\begin{thebibliography}{99}
\bibitem{DH} C. Dafermos,  W. Hrusa. {\em Energy methods for quasilinear hyperbolic initial-boundary value problems. Applications to elastodynamics.} Archive for Rational Mechanics and Analysis 87.3 (1985): 267-292.

\bibitem{Den} N. Dencker. {\em On the propagation of polarization sets for systems of real principal type.} Journal of Functional Analysis 46.3 (1982): 351-372.

\bibitem{DUW} M. de Hoop, G. Uhlmann, Y. Wang. {\em Nonlinear responses from the interaction of two progressing waves at an interface.} arXiv:1708.07578  (2017). 

\bibitem{Du} J. J. Duistermaat. {\em Fourier Integral Operators.} Vol. 130. Springer Science \& Business Media, 1995.

\bibitem{Gol} Z. A. Gol'dberg. {\em Interaction of plane longitudinal and transverse elastic waves.} Soviet Physics Acoustics 6.3 (1961): 306-310.

\bibitem{GrU93} A. Greenleaf, G. Uhlmann. {\em Recovering singularities of a potential from singularities of scattering data.} Communications in Mathematical Physics 157.3 (1993): 549-572.

\bibitem{HU} S. Hansen, G. Uhlmann. {\em Propagation of polarization in elastodynamics with residual stress and travel times.} Mathematische Annalen 326.3 (2003): 563-587.

\bibitem{JK} G. L. Jones, D. R. Kobett. {\em Interaction of elastic waves in an isotropic solid.} The Journal of Acoustical Society of America Vol 35, no. 1, 1963.

\bibitem{Ho3} L. H\"ormander. {\em The analysis of linear partial differential operators III: Pseudo-differential operators.}  Reprint of the 1994 edition. Classics in Mathematics. Springer, Berlin, 2007. 

\bibitem{Ho4} L. H\"ormander. {\em The analysis of linear partial differential operators IV: Fourier integral operators.} Reprint of the 1994 edition. Classics in Mathematics. Springer-Verlag, Berlin, 2009. 

\bibitem{HKM} T. Hughes, T. Kato, J. E. Marsden. {\em Well-posed quasi-linear second-order hyperbolic systems with applications to nonlinear elastodynamics and general relativity.} Archive for Rational Mechanics and Analysis 63.3 (1977): 273-294.

\bibitem{KSC} B. N. Kuvshinov,  T. J. H. Smit,  X. H. Campman. {\em Non-linear interaction of elastic waves in rocks.} Geophysical Journal International 194.3 (2013): 1920-1940.

\bibitem{LL} L. Landau, E. M. Lifshitz. {\em Theory of Elasticity} Vol. 7. Course of Theoretical Physics 3 (1986): 109.

\bibitem{LUW1} M. Lassas, G. Uhlmann, Y. Wang. {\em Inverse problems for semilinear wave equations on Lorentzian manifolds.} 
Communications in Mathematical Physics.

\bibitem{Me} R. Melrose. {\em Differential analysis on manifolds with corners.} Unpublished lecture notes.

\bibitem{MU} R. Melrose, G. Uhlmann. {\em Lagrangian intersection and the Cauchy problem.} Communications on Pure and Applied Mathematics 32.4 (1979): 483-519.

\bibitem{NW} G. Nakamura, M. Watanabe. {\em An inverse boundary value problem for a nonlinear wave equation.} Inverse Problems and Imaging 2.1 (2008): 121-131.

\bibitem{No} A. N. Norris. {\em Finite-amplitude waves in solids.} Nonlinear Acoustics 9 (1998): 263-277.

\bibitem{Ra0} L. Rachele. {\em Boundary determination for an inverse problem in elastodynamics.} Communications in Partial Differential Equations 25.11-12 (2000): 1951-1996.

\bibitem{Ra} L. Rachele. {\em An inverse problem in elastodynamics: uniqueness of the wave speeds in the interior.} Journal of Differential Equations 162.2 (2000): 300-325.

\bibitem{SD} C. Stolk,  M. V. de Hoop. {\em Microlocal analysis of seismic inverse scattering in anisotropic elastic media.} Communications on Pure and Applied Mathematics 55.3 (2002): 261-301.

\bibitem{Tay} M. E. Taylor. {\em Reflection of singularities of solutions to systems of differential equations.} Communications on Pure and Applied Mathematics 28.4 (1975): 457-478.

\bibitem{WZ} Y. Wang, T. Zhou. {\em Inverse problems for quadratic derivative nonlinear wave equations.} Communications in Partial Differential Equations (to appear).


\end{thebibliography}
\end{document}